\documentclass{amsart}
\usepackage{amsfonts,amsmath,amssymb,amscd}
\usepackage{graphicx}
\usepackage{varioref}
\usepackage{color}   
\usepackage[backref=page]{hyperref}
\hypersetup{
    colorlinks=true, 
    linkcolor=red,  
} 

\newcommand{\bea}{\begin{eqnarray}}
\newcommand{\ena}{\end{eqnarray}}
\newcommand{\beas}{\begin{eqnarray*}}
\newcommand{\enas}{\end{eqnarray*}}
\newtheorem{theorem}{Theorem}
\newtheorem{corollary}[theorem]{Corollary}
\newtheorem{conjecture}[theorem]{Conjecture}
\newtheorem{proposition}[theorem]{Proposition}
\newtheorem{lemma}[theorem]{Lemma}
  
\newtheorem{exercise}[theorem]{Exercise} 

\theoremstyle{definition}
\newtheorem{example}[theorem]{Example}

\def\bex{\begin{example}}
\def\eex{\end{example}}
\numberwithin{theorem}{section}
\numberwithin{figure}{section}                

\newcommand{\ignore}[1]{}
\def\bbbone{1}

\def\indist{=^d}
\def\todist{\Rightarrow}
\def\bbbe{\mathbb{E}}
\def\bbbe{\mathbb{E}\mkern1mu}   
\def\bbbp{\mathbb{P}}
\def\bbbr{\mathbb{R}}
\def\BX{{\bf X}}
\def\BY{{\bf Y}}
\def\BZ{{\bf Z}}
\def\BXi{{\bf X}^{(i)}}

\def\p{\bbbp}   
\def \e {\bbbe}   
\def\BR{\bbbr}
\def\bbbz{\mathbb{Z}}
\def\bbbn{\mathbb{N}}
\def\levy{L\'evy }
\def\cF{\mathcal{F}}
\def\cL{\mathcal{L}}
\def\cT{\mathcal{T}}
\def\t{{\ell}}  
\begin{document}

\title
{Size bias for one and all}

\author[Arratia]{Richard Arratia}
\address[Richard Arratia]{Department of Mathematics, University of Southern California,
Los Angeles CA 90089.}
\email{rarratia@usc.edu}

\author[Goldstein]{Larry Goldstein}
\address[Larry Goldstein]{Department of Mathematics, University of Southern California, Los Angeles CA 90089.}
\email{larry@usc.edu}

\author[Kochman]{Fred Kochman}
\address[Fred Kochman]{Center for Communications Research, 805 Bunn Drive
Princeton, NJ 08540}
\email{kochman@idaccr.org}

\date{September 7, 2018}

\begin{abstract}
Size bias occurs famously in waiting-time paradoxes, undesirably in sampling schemes, and 
unexpectedly in connection with Stein's method, tightness,  analysis of the lognormal distribution, Skorohod
embedding, infinite
divisibility, branching processes,
and
number theory. In this paper we review the basics and survey some of these unexpected connections.
\end{abstract}

\maketitle


\tableofcontents

\section{Prologue: the waiting time paradox}
\label{sec:waiting}

In the famous ``waiting time paradox'', see Feller \cite[Section I.4]{feller2},  there
are two plausible but conflicting analyses of the waiting time for the
next bus, once you get to the bus stop. More formally, this paradox concerns
the waiting time $W_t$ for the next arrival, 
starting from an arbitrary instant $t$, in a standard homogeneous Poisson process with intensity 
parameter $\lambda=1$:
(a) The lack of memory of the exponential interarrival time 
suggests that 
$\e W_t$ is not sensitive to the choice of $t$; so $\e W_t = \e W_0 =1$.  \ (b) Since the 
starting time is chosen uniformly in the interval between two successive arrivals, an interval of 
mean length 1, symmetry suggests that $\e W_t = 1/2$. 

As Feller shows, the reasoning behind \emph{both} analyses is faulty, because it is the 
instant and not the 
interval which is arbitrary: a longer interval thereby becomes more likely than the 
relative frequencies of interarrival lengths would suggest, a canonical instance of size biasing. 
So an unqualified appeal to properties of the original interarrival distribution is fallacious. 

In fact, as we will discuss, a reasonable but precise interpretation of ``arbitrary instant'' 
leads to the answer given in (a), though not for the reason given in (a). 
 
Not just recreational chestnuts, but also practical matters, such as statistical sampling tasks, are bedeviled by size 
bias; we provide a few references later.
Surprisingly, however, size bias  
plays a role in such unexpected contexts as  Stein's method, 
Skorohod embedding, nonuniqueness in the method of moments, infinite
divisibility of distributions, branching processes, and number theory. 
We will return to 
the ``paradox'' shortly, after giving the basics of size bias. Then we will survey size
bias as it appears in some of the non-sampling 
contexts.\footnote{An early draft of the present paper, with the title `Size biasing, when is the increment independent?', 
has been circulated since 1998, and was cited in \cite{pitman03}; an update,
`Size bias, sampling, the waiting time paradox, and infinite
divisibility: when is the increment independent?'
was cited in \cite{pekoz,archimedes,ross1}.
Both of these drafts are superseded  
 by this
paper.}

In \cite[pp. 78--80]{abt}, the authors introduce their two and one half page survey of size bias by saying ``Size-biasing arises naturally in
statistical  sampling theory (cf. Hansen and Hurwitz (1943) \cite{hh},
Midzuno (1952) \cite{midzuno} and Gordon (1993) \cite{gordon}), and
the results we present below are all well known in the folk
literature.''
In the present paper, we feel that we have contributed
a number of new results:
the conceptual heuristic given in Section
 \ref{sec:renewal} to explain  \eqref{noniid}, where a sum of independent variables is size biased by biasing only a single term,
 the explanation of an intimate
connection between uniform integrability and tightness in Section
\ref{sect UI},  
the size bias perspective on Skorohod embedding in \ref{sect skorohod}, and the treatment of infinite
divisibility in Section  \ref{sect inf div}
--- at least 
the argument based on \eqref{inf div sum},  size biasing a sum by size biasing a single
summand.

Another survey of size bias, with a different focus,  is \cite{brown}.

\section{Size bias basics}\label{sect basics}

\subsection{Bias in general}\label{bias general}

Let $h$ be a nonnegative function, and $X$ be a random variable taking
values in the domain of $h$, with   $ \e h(X) \in (0,\infty)$. 
For such $X$ and $h$, we say $X^h$ has the $h$-biased $X$ 
distribution if and only if the distribution of $X^h$, relative to
the distribution of $X$, has
Radon-Nikodym derivative given by
\bea
\label{h bias}
\frac{\p(X^h \in dx)}{\p(X \in dx)} = \frac{h(x)}{\e h(X)}.
\ena
The support of (the distribution of) $X^h$ is then a subset of the
support of $X$, possibly a proper subset due to the set  where $h$=0:
\begin{equation}\label{support}
\text{supp}(X^h)=(\text{supp}(X) \setminus
  h^{-1}(0))^{\text cl},
\end{equation}
where $A^{\text cl}$ denotes the closure of $A$.
A nice pair of examples,
both having equal support for $X$ and $X^*$, and 
 using $h(x)=x$, is presented in Figure \vref{figure cantor}.

The class of exponential functions, $h(x)=e^{\beta \, x}$ for various choices of $\beta \in
(-\infty,\infty)$, is very important.  This class is  central to exponential families and large deviation
theory, but no single value $\beta$ plays a special role.
The family of power functions $h(x)=x^\beta$ for $\beta>0$ might be viewed as  runner up, behind the family of exponential functions, but here the  choice $\beta=1$ is truly special. We believe that $h(x)=x$  for $x \ge 0$ is \emph{the most important} example of bias.

\subsection{Size bias in particular}\label{size particular}
 
When $h$ is the function $h(x)=x$ with domain $[0,\infty)$, the
$h$-bias above is called size bias.  Thus, one can size bias the
distribution of any nonnegative random variable $X$ for which $a := \e
X \in (0,\infty)$.   Instead of $X^h$ one writes $X^*$ or $X^{\rm s}$ 
for a random
variable with the size-biased distribution of $X$. The
characterization \eqref{h bias} reduces to
\bea
\label{sbdF}\label{size bias RN}
\frac{\p(X^* \in dx)}{\p(X \in dx)} = \frac{x}{a}.
\ena

For the common special cases, where $X$ is discrete with probability
mass function $f$, or where $X$ is 
absolutely continuous
with density $f$, the formula
\bea \label{sizebias-f}
f_{X^*}(x)=\frac{xf(x)}{a},
\ena
completely specifies the  size-biased distribution.

Does size bias commute with conditioning, on events of the form $(X \in B)$?   The answer, of course, is yes --- provided  that   $\p(B)>0$. 
This is made obvious using the bias-in-general viewpoint of Section
\ref{bias general}:  any two biasings commute, because  multiplication is commutative.  In detail:  
suppose that $g$, like $h$ in Section \ref{bias general}, is a nonnegative function whose domain includes the support of $X$, and $\e g(X) \in (0,\infty)$.   Then one can bias with respect to $g$, to specify the distribution of $X^g$.  Elementary conditioning, on the event $(X \in B)$, is precisely the case where $g$ is the indicator function for $B$;  in this case $\e g(X)=\p(B)<\infty$, and  in the phrase \emph{elementary conditioning}, the word elementary means that $\p(B)>0$.  Back to the general case: suppose that the product $gh$ (the pointwise product, 
not the composition $(g \circ h)(x)=g(h(x))$,) also has strictly positive, finite expectation, i.e.,  $\e (g(X) h(X) ) \in (0,\infty)$.  Then the iterated biased distributions,  of $(X^h)^g$, 
and of $(X^g)^h$ are, of course, equal to each other,
since they are both the same as the distribution of $X^{(gh)}$.

An interesting case of \eqref{sizebias-f} involves the Poisson
distributions.   Starting from the assumption that the distribution of $X$ is
Poisson ($\lambda)$,
so that $f(k) = e^{-\lambda}\lambda^k/k!$, then \eqref{sizebias-f}
with $x=k+1$ for $k=0,1,2,\ldots$ gives
\begin{equation}\label{baby poisson}
f_{X^*}(k+1)=\frac{(k+1)f(k+1)}{\lambda} = \frac{k+1}{\lambda}
e^{-\lambda} \frac{\lambda^{k+1}}{(k+1)!}=f(k).
\end{equation}
Hence for $X$ with a Poisson ($\lambda)$ distribution, $0 < \lambda < \infty$,
\begin{equation} \label{add one}
X^* \indist X+1.
\end{equation}
The above result is sometimes called \emph{Robbins' Lemma}.


Conversely,
\begin{proposition}\label{proposition baby}
Suppose that  $X$ is a
nonnegative \emph{integer-valued} random variable with mean $\lambda
\in (0,\infty)$, and   $X^* =^d X+1$.  Then  $X$ is Poisson
($\lambda)$.
\end{proposition}
\begin{proof}
Equation \eqref{sizebias-f} shows that for $k \ge 0$, the point mass function
$f$ for $X$ satisfies $f(k+1)=\lambda \, f(k)/(k+1)$, hence by
induction $f(k) = f(0) \lambda^k/k!$.  The assumption that  $X$ is 
nonnegative \emph{integer-valued} implies that $\sum_{k \ge 0} f(k)
=1$, hence $1=\sum_{k \ge 0} f(0) \lambda^k/k! = f(0) e^\lambda$.
\end{proof}
The observation that  a nonnegative integer-valued random variable $X$
is Poisson($\e X$) if and only if $X^* =^d X+1$ may be viewed as the starting point for
Stein's method for the Poisson distribution;  see Section \ref{sect chen}.

It is also true that if $X\ge 0$ and $0 < \lambda:= \e X < \infty$ and $X^* =^d
X+1$, (without assuming that $X$ is  \emph{integer-valued},) then $X$
is Poisson ($\lambda)$. This is not so obvious, and the reader
might enjoy giving his own elementary proof, by combining the support
consideration \eqref{support} with \eqref{sizebias-f};  alternately, see   Theorem
\ref{thm inf div and Levy} and Corollary \ref{cor poisson}.

If $X$ is Bernoulli($p$), meaning that $\p(X=1)=p, \p(X=0)=1-p$, and if $p>0$, then $X$ can be size biased.  Using either the support consideration \eqref{support}, or the mass function formula 
\eqref{sizebias-f}, we see that 
\begin{equation}\label{bernoulli}
\text{ for } X \sim \text{Bernoulli}(p), 0 < p \le 1, \ \
X^* =1.
\end{equation}
This Bernoulli family example shows that the size bias transformation is not one to one. 

It is easy to see that (\ref{sbdF}) is implied by
\begin{equation}\label{sizebias continuous}
\text{ for all bounded \emph{continuous} } g, \ \ \e  g(X^*) =  \frac{1}{a} \ \e (X g(X) ),
\end{equation}
and that \eqref{sbdF} implies\footnote{Some would say, by definition,  \eqref{sbdF} means exactly
\eqref{sizebias},  others might say that by definition,  \eqref{sbdF} means 
\eqref{sizebias} restricted to $g$ being the indicator function of a measurable subset of $[0,\infty)$.}
\begin{equation}\label{sizebias}
\text{ for all bounded \emph{measurable} } g, \ \ \e  g(X^*) =  \frac{1}{a} \ \e (X g(X) ).
\end{equation}

Even in the discrete and absolutely continuous cases,
where the elementary identity \eqref{sizebias-f} applies, the characterization
of size bias via \eqref{sizebias} is very handy for manipulations.

\subsubsection{Generating functions}
Let $\phi \equiv \phi_X$ be the characteristic function of $X$, so
that $\phi(u) = \e e^{iuX}$.  A standard fact, for example from
\cite[XV.4 Lemma 2]{feller2},  is that $\e |X| < \infty$ implies $\phi$
is differentiable, with $\phi'(u) =  i\e (Xe^{iuX})$.  With
$g(x)=e^{iux}$,
by taking real and imaginary parts, \eqref{sizebias} implies that for
nonnegative $X$ with $a := \e X \in (0,\infty)$,
\begin{equation}\label{phi'}
    \phi_{X^*}(u) :=   \e e^{iuX^*} =  \frac{1}{a}  \e (Xe^{iuX}) =
    \frac{1}{i \,  a} \phi_X'(u),
\end{equation}
and since characteristic functions determine distribution,
\eqref{phi'} also completely specifies the size bias distribution.
Suppressing the dummy variable to get a clean display, \eqref{phi'}
says
$$  \phi_X' = i \  \e X \ \phi_{X^*} \ .
$$

In case the nonnegative random variable $X$ above is integer valued, one
could use probability generating functions instead of characteristic
functions to characterize the distributions of $X$ and $X^*$.  With random
variable name $N$ in place of $X$, and $p_n := \p(N=n)$, we have
$G_N(z)=\sum p_n z^n$ with derivative $G_N'(z) = \sum n p_n z^{n-1}$.
So if $0 < \e N < \infty$, the generating function for the size biased random variable is
$$
   G_{N^*}(z) := \sum_{n \ge 0} \p(Z^*=n) z^n = 
 \sum \frac{n p_n}{\e N} z^n = \frac{z}{\e N} \ G_N'(z); 
 \ G_{(N^*-1)}(z) = \frac{1}{\e N} \ G_N'(z) \ .
$$
Suppressing the dummy variable to get a clean display, this relation is
\begin{equation}\label{G'}  
G_N' = \e N \ G_{(N^*-1)} \ .
\end{equation}

\subsubsection{Compound distributions for random sums}
Here is an application of \eqref{phi'}.
Suppose $N$ is a nonnegative integer valued random variable, with
finite strictly positive mean, and $X,X_1,X_2,\ldots$ are i.i.d.,
independent of $N$.  With $S_n := X_1+\cdots+X_n$ and $Z=S_N =
X_1+\cdots + X_N$, the
distribution of the random sum $Z$ is called a \emph{compound
  distribution},  although the phrase itself is often used to refer to
mixtures in general.  With the notation $G$ for 
probability generating functions, 
and $\phi$ for
characteristic functions, 
the characteristic function of $Z=X_1+\cdots+X_N$ is $\phi_Z = G_N \circ
\phi_X$.  Now if $X \ge 0$ and $0 < \e X < \infty$, so that $X$ can be
size biased, then $Z$ can also be size biased. From \eqref{phi'} we
have
$$
    \phi_{Z^*}(u)  =
    \frac{1}{i \, \e Z} \phi_Z'(u) =
    \frac{1}{i \, \e Z} (G_N(\phi_X(u)))' 
$$
so from the chain rule, and $\e Z = \e N \, \e X$, and \eqref{G'},  we have 
\begin{eqnarray}\label{phi' compound}
    \phi_{Z^*}(u) & = &     \frac{1}{i \, \e Z} \times G_N'(\phi_X(u)) \times \phi_X'(u) \nonumber  \\
      & = &   \frac{1}{i \, \e N  \, \e X} \times \e N
    \   G_{(N^*-1)}(\phi_X(u)) \times i\  \e X \ \phi_{X^*}(u) \nonumber \\
 & = &  G_{(N^*-1)}(\phi_X(u)) \times  \phi_{X^*}(u).   \label{compound bias}
\end{eqnarray}
Since a product of two characteristic functions gives the distribution
of the sum of two independent random variables,  \eqref{compound bias}
specifies a rule:

\noindent {\bf Rule for size biasing a random sum}.  To size bias a
random sum $Z=S_N$ of $N$ independent copies of $X$, where both $N$ and $X$ are
nonnegative, with strictly positive finite mean:  (1) Size bias  $N$
to get $N^*$, and then take
\emph{one less},
 $N^* -1$, as the new compounding variable, and (2) Add in an
independent copy of the size biased version of the summand, $X^*$.
To say the same, less formally:  size bias the number of summands, and
replace one summand by a size biased version.

Of course, specializing $N$ to be concentrated at a fixed positive integer $n$ immediately
yields a rule for size biasing a sum of $n$ independent identically
distributed terms.  
However, we will rederive that rule, as \eqref{iid} 
in Section \ref{sect bias sum} below, 
which studies how to size bias a sum of (a nonrandom number of)
random variables, with no need to assume either independence or
identical distribution for the summands.

\subsubsection{Unbounded functions, and moments}
Recall that for a real valued random variable, ``$\e Y \in [-\infty,\infty]$ exists'' means that it is \emph{not} 
 the case that \emph{both} the positive and the negative parts of $Y$ have infinite expectation.  We extend slightly the 
 statement that, if $\e |X g(X)|< \infty$, then $\e g(X^*) =  \ \e (X g(X) )/\e X$.
 
\begin{lemma}\label{lemma monotone}
Let $g: [0,\infty) \to \BR$ be measurable, and let $X$ be a nonnegative random variable with $a := \e X \in (0,\infty)$.
\begin{equation}\label{dominated}
\text{ If } \e (X g(X)) \in [-\infty,\infty] \text{ exists},  \text{ then} \ \ \e  g(X^*) =  \frac{1}{a} \ \e (X g(X) ).
\end{equation}
If $\e (X g(X))$ doesn't exist in $[-\infty,\infty]$, then neither does 
$\e g(X^*)$.
 \end{lemma}
 \begin{proof}
 In outline, the proof is:  consider separately the positive and negative parts of $g$;   for each of these,
 apply \eqref{sizebias} to truncations, and apply monotone
 convergence.

   In detail:
when $g(x) \ge 0$, by applying \eqref{sizebias} to
$g_n(x)=\max(g(x),n)$, and taking limits, we conclude that
$$
 \e  g(X^*) =  \frac{1}{a} \ \e (X g(X) )
$$
holds, \emph{including} the case where both sides are infinite.
Write $y_+$ and $y_-$ for the positive and negative parts of $y$.
Then the functions $g_+$ and $g_-$ given by $g_+(x) = (g(x))_+$ and $g_-(x)=(g(x))_-$ 
are nonnegative. 
Note that on the domain $[0,\infty)$, $(x g(x))_+ = x g_+(x)$ and $(x g(x))_- = x g_-(x)$.
Under the hypothesis that $\e (X g(X)) \in [-\infty,\infty]$ exists,  
at least one of $h = g_+$ and $h=g_-$ has $ \e (X h(X)) < \infty$, and hence $\e g(X^*) = \e g_+(X^*) - \e g_-(X^*) \in [-\infty,\infty]$ is well defined, with value given by $(1/a) \e( Xg_+(X)) -(1/a) \e(X g_-(X)) 
= (1/a) \e (X(g(X))$.   Likewise, when $\e ((Xg(X))_+) = \e ((Xg(X))_-) = \infty$, we have  both $ \e g_+(X^*) = \e g_-(X^*) =\infty$ so that $\e g(X^*)$ does not exist.
\end{proof}

In particular,
taking
  $g(x)=x^n$ in \eqref{dominated}, we have
\begin{equation}\label{moment shift}
\e (X^*)^n=\e X^{n+1}/\e X
\end{equation}
and this includes the case where both sides are infinite.
 Apart from the extra scaling by $1/\e X$, \eqref{moment shift} says that the sequence of moments
of $X^*$ is the sequence of moments of $X$, but shifted by one.
Hence one  way to recognize size biasing is through
the  \emph{shift of the moment sequence}; this plays a role in two interesting examples,
\eqref{borel moments} and \eqref{lognormal moment}.

\subsubsection{Stochastic monotonicity}
It is easy to see that, in general, $X^*$ lies
above $X$ in distribution, i.e., $\p(X^*>t) \geq \p(X>t)$ for all $t$.
In detail:  
letting $g(x) = \bbbone(x>t)$ in (\ref{sizebias}) for some fixed $t$, 
\begin{equation}\label{lies above}
\p(X^*>t)= \frac{\e (X\bbbone(X>t))}{\e X}  \ \ge \  \frac{ \e X \  \e  \bbbone(X>t)}{\e X} =\p(X>t)
\end{equation}
where the  inequality above is the special case $f(x)=x$, $g(x) =\bbbone(x>t) $
of Chebyschev's correlation inequality: $\e (f(X) g(X)) \geq
\e f(X) \ \e g(X)$ for any random variable and any
two increasing functions $f,g$.

The condition $\p(X^*>t) \geq \p(X>t)$ for all $t$ is described as ``$X^*$ lies
above $X$ in distribution,'' written $X \le_{st} X^*$, and implies  that there exist couplings of
$X^*$ and $X$ in which always $X \le X^*$.  Writing $Y$ for  the
difference, we have
 \begin{equation}
\label{xplusy}
X^*=X+Y, \ \ Y \geq 0.
  \end{equation}
In general, the known marginals for $X$ and $X^*$ do not uniquely determine the distribution of a coupling; in Section \ref{sect inf div} we will study the question:  when can \eqref{xplusy} be achieved with $X,Y$ independent?  

Suppose the distribution of $Z$ is defined to be that of $X$, conditional on $(X>0)$.  Recalling the third paragraph
of Section \ref{size particular}, it is obvious that $X^* =^d Z^*$.   And of course, 
$Z$ lies above $X$ in distribution
since for $t \ge 0$, $\p(X>t|X>0) = \p(X>t)/\p(X>0) \ge \p(X>t)$. 
To summarize, for nonnegative $X$ with $\e X \in (0,\infty)$, we have the stochastic monotonicity sandwich
$$    
     X   \le_{st} (X|X>0) \le_{st} X^*.
$$

\subsubsection{Scaling, coupling, and limits in distribution}
It is easy to see, from \eqref{sizebias}, that size biasing respects
multiplication by positive constants, that is, with  $c>0$, 
\begin{equation}\label{scaling}
(cX)^*=^d c(X^*).
\end{equation}
The notation used above, $X =^d Y$, is often written 
${\mathcal  L}(X)={\mathcal L}(Y)$, to say that random variables $X$ and $Y$
have the same law, or distribution. The simpler notation $X=Y$ would
imply a coupling, i.e., that $X$ and $Y$ are defined on the same
probability space, with $X(\omega)=Y(\omega)$ for all outcomes $\omega$. 
\ignore{
JOKE FOR MY COAUTHORS:
As an exercise in notation: scaling also respects a size bias coupling:  if $X,Y$ are defined on the same probability space, and
$X^* =^d Y$, then $cX, cY$ are also defined on the same probability space (with each other, and of course also with $X,Y$),
and $(cX)^* =^d cY$.   The first vector named, $(X,Y)$, is a size bias coupling, and its scaled multiple, $c(X,Y)$, is also a size bias coupling.
SERIOUSLY, should we add some comment that "All three authors of this paper have been bedeviled by the issue, when to write $=$ and when to write $=^d$.
}

It is also true that size bias respects convergence in distribution,
provided one is careful to make the additional hypothesis that the means converge to the mean of the limit random variable, which is in this context equivalent to uniform integrability.   
\begin{theorem}\label{limit theorem}
Suppose that $X,X_1,X_2,\ldots$ are nonnegative random variables with
$a:= \e X \in (0,\infty)$, $a_n := \e X_n \in (0,\infty)$, that  $X_n \Rightarrow X$, and that $a_n \to a$.  Then
$$
    X_n^* \Rightarrow X^*.
$$
\end{theorem}
\begin{proof}
Let  \mbox{$h: \BR \to \BR$} be a 
bounded continuous function  with compact support.  
Then the function $g$ given by $g(x)=x
\, h(x)$ is bounded and continuous.  Since  $g$ is bounded, 
\eqref{sizebias} applies, and since $g$ is continuous, the
hypothesized distributional convergence implies $\e g(X_n) \to \e
g(X)$. Using \eqref{sizebias} with $h$ in the role of $g$, we have
$$
   \e h(X_n^*)= \frac{\e X_n h(X_n)}{a_n} = \frac{\e g(X_n)}{a_n} \to
   \frac{\e g(X)}{a} = \frac{\e X h(X)}{a} = \e h(X^*).
$$  
\end{proof}

The necessity of the hypothesis that $\e X > 0$, in Theorem \ref{limit theorem}, is shown  by the example with $X_n$ distributed as  
Bernoulli($1/n$), so that $X_n^* \Rightarrow 1$,
and  $X_n \Rightarrow X = 0$, but the limit random variable 
$X$ cannot be size biased.

The converse of Theorem \ref{limit theorem} is 
false, since
the correspondence ${\mathcal L}(X) \mapsto {\mathcal L}(X^*)$ is many
to one. In detail,  take any $A,B$ with  $A \ne^d B$ and $A^* =^d B^*$; then the sequence $X_1,X_2,X_3,X_4,\ldots = A,B,A,B, \ldots$, together with $X=A$, has $X_n^* \Rightarrow X^*$ but not $X_n \Rightarrow X$. 

An interesting natural example, related to the non-converse of Theorem \ref{limit theorem},
involves $X_n$ which cannot be rescaled to have a nontrivial limit distribution, while the corresponding $X_n^*$ can.
Take
$X_n$ to have the Borel distribution\footnote{For $\lambda \in [0,1)$, 
one says $X$ has the Borel($\lambda$) distribution if $X$ is the total progeny in the subcritical Galton-Watson branching process where the individual offspring distribution is Poisson with mean $\lambda$, equivalently,
$\p(X=i)=\exp(-\lambda \, i) (\lambda \, i)^{i-1}/i!$ for $i=1,2,\ldots$; see
\cite{aldous-pgw}.}  with parameter $\lambda=1-1/n$.  Calculation shows that $\e (X_n) = n$ and for $k=1,2,\ldots,$ 
$\e (X_n)^{k+1} \sim n^{2k+1} (2k-1) (2k-3) \cdots 5 \times 3 \times 1$,  
hence one cannot scale the $X_n$ 
sequence
to get a nontrivial distributional
limit.  But using \eqref{moment shift}, we have 
\begin{equation}\label{borel moments}
\e (X_n^*)^k \sim 
 n^{2k} (2k-1) (2k-3) \cdots 5 \times 3 \times 1, 
 \end{equation}
 so that with $Z$
 for a standard normal, $X_n^*/n^2 \Rightarrow Z^2$.

\subsubsection{Mixtures, biasing a conditional probability}


First, we give Lemma \ref{lemma mixture}, an elementary result on how to size bias a mixture of distributions. An application of Lemma 
\ref{lemma mixture}  will be given by Lemma \ref{lemma Leipnik}  and the subsequent 
Theorem \ref{thm ra}.  Mixtures are often discussed in conjunction with regular conditional probabilities; see for example \cite{chow-teicher,loeve}.

Suppose that $I \subset \BR$, that $h$ is a probability measure on $I$, that for each $b \in I$ $\mu_b$ is a distribution for
a nonnegative random variable $X_b$, with $m(b) := \e X_b \in (0,\infty)$, and that $b \mapsto \mu_b$ is measurable.  Note, we have assumed that for every $b$, $m(b) \in (0,\infty)$ in order that, for every $b$, 
the size-biased distribution for $X_b^*$ be defined.  We say that
(the distribution of) $X$ is the mixture of (the distributions of) $X_b$, governed by $h$, if for all bounded measurable $g$,
$$
   \e g(X) = \int  \e g(X_b) \ dh(b).
$$   
Of course, for such a mixture, $\e X = \int m(b) \ dh(b) \in (0,\infty]$, but since we are interested is size bias, we make the additional
assumption that $a := \e X < \infty$.
\def\hs{h^s}
\begin{lemma}\label{lemma mixture}
Under the setup of the previous paragraph, with $a = \int m(b) \ dh(b) \in (0,\infty)$, the distribution of $X^*$ is a mixture of the distributions of the $X_b^*$.  
The measure $\hs$ governing this mixture is defined in terms of the original governor $h$ via its Radon-Nikodym derivative, $d\hs(b)/dh(b) = m(b)/a$.  In particular, if
$m(b)$ is constant, then $\hs=h$, i.e., the measure  governing  $X^*$ as a mixture of the $X_b^*$ is equal to the measure governing $X$
as a mixture of the $X_b$.
\end{lemma}
\begin{proof}
For bounded measurable $g$
$$
   \e g(X^*) = \frac{\e( X g(X))}{a}  =  \int \frac{\e (X_b g(X_b))}{m(b)} \ \frac{m(b)\  dh(b)}{a} = \int  \e g(X_b^*) \ d\hs(b).
$$   
\end{proof}

In a different direction,
the following result from \cite{penrose} 
can be useful for constructing size bias couplings for continuous random variables that are not represented as sums, 
though it may also be noted 
that Lemma \ref{lemma 2.1}  implies 
\eqref{noniid}
for sums of indicator variables, see \cite[Lemma 2.6 ff]{BGI}.

\begin{lemma}\label{lemma 2.1}
Let $X = \Pr(A|\mathcal{F})$ where $\mathcal{F}$ is some $\sigma$-algebra and $A$ is some
event with $0<\Pr(A) <1$. Then $X^*$ has the distribution of $X$
conditioned on $A$. 
\end{lemma}
\begin{proof}
For any bounded measurable $g$, we have 
$$
\e g(X^*) = \frac{\e(g(X)\e(1_A|\mathcal{F}))}{\e X} =
\frac{\e ( \e( g(X)   1_A|\mathcal{F}))}{\p(A)} $$ $$
=\frac{ \e(g(X)1_A)}{\p(A)}=  \e[g(X)|A].
$$
\end{proof}

\subsubsection{Many to one, one to one}
We describe the preimage, under size biasing, of a random variable
$Z$.  Note first that if $Z =^d X^*$, then for any mixture $\mathcal{M}
= b \delta_0 +(1-b)\mathcal{L}(X)$
with $0\le b <1$, a random variable $Y$ with $\mathcal{L}(Y) =\mathcal{M}$ 
is also a preimage.  We claim that changing the amount of point-mass at $0$
is the only source
of non-uniqueness.

\begin{lemma}\label{lemma biject}
A random variable $Z$ satisfies $Z=^d X^*$ for some $X$ iff $1=\p(Z>0)$
and $\e (1/Z) < \infty$, 
and then there is a unique law for $Y>0$ such that any $X$ having $X^*=^dZ$ 
is distributed as $b\delta_0 +(1-b)\mathcal{L}(Y)$ for some $0\le b
<1$.
\end{lemma}

\begin{proof}
Let $Z =^d X^*$ for some 
$X$;  this implies $X \ge 0$,
$0< \e X < \infty$, and \mbox{$\p(Z>0)=1$}.  Let $b := \p(X=0)$, so clearly $b \in [0,1)$.
Let $Y$ have the distribution of $X$ conditioned
on $X>0$, so $Y>0$, $Z =^dY^*$,
 and $\mathcal{L}(X) = b\delta_0+(1-b)\mathcal{L}(Y)$.  
With $c = \e X / (1-b) = \e Y \in (0,\infty)$, 
we have, as in \eqref{size bias RN}, that the distributions $\nu$ of $Z$ and
$\mu$ of $Y$, as measures on $(0,\infty)$, are mutually absolutely
continuous, with Radon-Nikodym derivative
$$
\frac{\nu(dx)}{\mu(dx)} \equiv \frac{\p(Z \in dx)}{\p(Y \in dx)} = \frac{x}{c}.
$$
This shows the uniqueness of the law for $Y$;  that $\e (1/Z) < \infty$
follows from the explicit calculation
$$   
  \e \frac{1}{Z} = \int_{0<x<\infty} \frac{1}{x} \ \nu(dx) = \int
  \frac{1}{x} \ \frac{d \nu}{d \mu} \ \mu(dx) =  \int \frac{1}{c} \ \mu(dx)=\frac{1}{c}.
$$

Conversely, if  $Z>0$ with probability measure $\nu(dz)$ satisfies 
$0 < \e (1/Z) < \infty$, then  with  $1/c = \e (1/Z)$, the law $\mu$ on
$(0,\infty)$ with $ \mu(dy) /  \nu(dy) = c/y$, as the distribution for
$Y$, yields 
$Z =^d Y^*$. 

\end{proof}

A paraphrase of Lemma \ref{lemma biject} is that size bias is a
bijection, between equivalence classes of distributions for nonnegative 
random variables with strictly positive finite mean, modulo varying
the size of the point mass at zero, and distributions for strictly
positive random variables having finite minus first moment.

\subsection{To bias a process by one coordinate}\label{sect basics one}

The following is taken from \cite{GR96}.   Readers who dislike technicalities might prefer to jump directly to Section \ref{sect bias sum}, which leads up to 
\eqref{this identity}, and then come back only if they feel uncomfortable that our proof of 
\eqref{noniid infinite} doesn't involve any limits!
Suppose that $\BX  
= (X_1,X_2,\ldots) \in [0,\infty)^{\mathbb{N}}$ has joint law $\mu$, 
and for a particular choice of $i$, 
$a_i := \e X_i \in (0,\infty)$.  To bias by $X_i$ means, analogous to 
\eqref{size bias RN}, to switch to the joint law $\mu^{(i)}$ on $[0,\infty)^{\mathbb{N}}$ with Radon-Nikodym derivative
\begin{equation}\label{RN infinite}
\frac{d \mu^{(i)}}{d \mu}  = \frac{x_i}{a_i}.
\end{equation}
We write
${\bf X}^{(i)} = (X_1^{(i)},X_2^{(i)},\ldots)$ for a process having this joint distribution
$\mu^{(i)}$. 
Equivalent to \eqref{RN infinite} is the following
statement,
\begin{equation}\label{sizebias infinite}
\text{ for all bounded measurable } g, \ \ \e  g({\bf X}^{(i)}) =  \frac{1}{a_i} \ \e (X_i g(\BX) ),
\end{equation}
which looks very much like \eqref{sizebias}, except that now we have $g : [0,\infty)^{\mathbb{N}} \to \BR$.
Note that given a bounded measurable $h:  [0,\infty) \to \BR$,
applying \eqref{sizebias infinite} to the special case $g({\bf x}) := h(x_i)$ shows that our notion of  process bias by one coordinate, restricted to viewing that coordinate, agrees with
the original notion of size bias, i.e., $X_i^{(i)}=^d X_i^*$. 
In general, there is no similarly compact description of what happens to the other coordinates.
However, as we will see in Section \ref{Martingale size bias}, 
if the process ${\bf X}$ is a martingale then
biasing the process by any single coordinate results in size-biasing the marginal distribution of
{\em each} coordinate simultaneously.  


In a different direction, suppose that under  $\mu$ the coordinates are initially independent.  Then 
as we now show, after  biasing by the $i^{\rm th}$ coordinate they remain independent, and only the
$i$th coordinate is affected.
\begin{lemma}\label{lemma indep}
Fix a particular value $i$.  Assume that $X_1,X_2,\ldots$ are
mutually independent, nonnegative, and  that $0 < \e X_i < \infty$.  For $j \ne i$ let $Y_j =^d X_j$, 
let $Y_i =^d
X_i^*$, and let $Y_1,Y_2,\ldots$ be mutually independent.  
Then the law $\mu^{(i)}$ for $\BXi$, as given by \eqref{RN infinite}, reduces to the law
for ${\bf Y} = (Y_1,Y_2,\ldots)$, i.e.
$$
  (X_1^{(i)},X_2^{(i)},\ldots)  =^d (Y_1,Y_2,\ldots).
$$
\end{lemma}
\begin{proof}
First we check that the marginals match, i.e., that for each $j$, $X_j^{(i)} =^d Y_j$.  We already noted  that this is so, for $j=i$, as a consequence of \eqref{sizebias infinite}, even without the hypothesis of mutual independence.  For $j \ne i$, and a bounded measurable  $h:  [0,\infty) \to \BR$, applying \eqref{sizebias infinite} to the special case $g({\bf x}) := h(x_j)$ yields the relation  $\e g (\BXi) = \e h(X_j^{(i)})
= (1/a_i) \e (X_i h(X_j))$.  Using the independence of  $X_i$ and $X_j$, we get 
 $ \e h(X_j^{(i)})
= (1/a_i) \e (X_i h(X_j)) = (1/a_i) (\e X_i) \e h(X_j) = \e h(X_j)$, proving that for $j \ne i$, $ X_j^{(i)} =^d X_j$, as required, since for $j \ne i$, $Y_j =^d X_j$.

Next we show that $\BXi$ and $\BY$ have the same joint distribution,
either by showing that $\BXi$ has independent coordinates, or by
checking that for all measurable $C \subset [0,\infty)^{\mathbb{N}}$,
$\p(\BXi \in C) = \p (\BY \in C)$,   first by checking
finite-dimensional cylinder sets, then applying the $\pi - \lambda$
theorem  --- either route seems to require the same work.  Without
loss of generality, the cylinder set $C$ includes a restriction on the
$i^{\rm th}$ coordinate, i.e., it has the form $C =  (X_i \in B_i) \cap \bigcap_{j \in J} (X_j \in B_j)$, where $i \notin J$.  Write
$g_1({\bf x}) = 1(x_i \in B_i)$ and $g_2({\bf x}) = 1(x_j \in B_j \text{ for } j \in J)$.  
With $g=g_1 g_2$ in 
\eqref{sizebias infinite}, calculation that $\e g(\BXi) = \e g(\BY)$ is a simple extension of the calculation for the special case where the cylinder restricts only one coordinate, 
given in the first paragraph of this proof.
\end{proof}
  
Another technical issue involves the value infinity.   It would have been possible to 
present the basic discussion of size bias, in particular \eqref{size bias RN} and
\eqref{sizebias}, in terms of a random element $Y$ with values in $[0,\infty]$.  But since 
$0 < \e Y < \infty$ implies $\p(Y=\infty)=0$, it is of course
possible, and simpler, to deal with $Y$ taking values in $[0,\infty)$, 
and this is what everyone does.  However, in dealing with infinite
sums of finite nonnegative random variables, 
one cannot simply declare that the space of values for the sum be
taken as $[0,\infty)$, 
even if one knows that the sum is finite with probability one.

Our goal is to deal with the distribution of random variables $Y = h(\BX)$, such as
$Y=X_1+X_2+\cdots$,\footnote{Thanks to only having nonnegative numbers
  for the coordinates of the domain, 
there are no convergence issues in dealing with the sum $X_1+X_2+\cdots \in [0,\infty]$.} and to specify the distribution of $Y^{(i)}$, 
distributed as $Y$ with $\mu$ changed to $\mu^{(i)}$.  Hence we
consider measurable $h: [0,\infty)^{\mathbb{N}} \to [0,\infty]$,  
and  bounded measurable $f: [0,\infty] \to \BR$.  The composition
 $g(\BX) = f(h(\BX))$ is a bounded measurable function from 
$[0,\infty)^{\mathbb{N}} \to \BR$, hence \eqref{sizebias infinite}
applies.  
The distribution of  $Y^{(i)}$ is then specified by
\begin{equation}\label{sizebias functional}
\text{ for bounded measurable } f: [0,\infty] \to \BR,
\ \ \ \e (f(Y^{(i)})) =  \frac{1}{a_i} \ \e (X_i f(Y ) ).
\end{equation}
\ignore{
In particular, when $Y=S=X_1+\cdots+X_n$  or $Y=S= X_1+X_2+\cdots$, we take $h({\bf x})=x_1+\cdots+x_n$ or
$h({\bf x}) = x_1+x_2+\cdots$, as appropriate, and then
\eqref{sizebias functional} specifies the distribution of $S^{(i)}$, 
corresponding to the sum $S$  when the process $\BX$ is biased by
$X_i$.  
Less formally, one writes $S^{(i)}=X_1^{(i)}+\cdots+X_n^{(i)}$ or
$S^{(i)}=X_1^{(i)}+X_2^{(i)}+\cdots$.  
 Thanks to only having nonnegative numbers for the coordinates of the domain, there are no convergence issues in dealing with the sum $X_1+X_2+\cdots \in [0,\infty]$.
}

\subsection{To size bias a sum}\label{sect bias sum}

Consider a finite sum $S=X_1+\cdots+X_n$, $n \ge 1$, or an infinite sum $S=X_1+X_2+\cdots$, 
with  $X_i \ge 0$ and  $a_i := \e X_i >0$, and $a=\e S < \infty$.
After biasing by  $X_i$, as in \eqref{RN infinite}, 
we have a sum\footnote{Warning: our notation
here conflicts with 
some
standard expositions of Stein's method,  such as
\cite{chen},
\cite[Theorem B.1]{BHJ92}, and \cite{grimmett}, where notation $V_i$ refers to the sum,
with $i$th term omitted, size biased by the $i$th term.}
  $S^{(i)}
=X_1^{(i)}+\cdots+X_n^{(i)}$,
so that, as a special case of \eqref{sizebias functional},  for bounded nonnegative measurable $g$,
$$
  \e g(S^{(i)}) = \frac{1}{a_i} \e (X_i g(S)),
$$
and then with \eqref{sizebias} to justify the first line, and elementary algebra (here using $g \ge 0$) to justify the second line,
\bea
  \e g(S^*) &=&  \e (Sg(S))/a\nonumber \\
         &=& \sum_i \frac{1}{a} \ \e (X_i g(S)) \nonumber \\
         &=& \sum_i   \frac{a_i}{a}  \  \e g(S^{(i)}). \label{pre mixture}
\ena

Suppose  furthermore that the summands $X_1,X_2,\ldots$ are independent.  If
size biased random variables $X_1^*,X_2^*,\ldots$ are realized on the
same probability space, with $(X_1,X_1^*),(X_2,X_2^*),\ldots$ mutually\
independent, then for each  $i$, by Lemma \ref{lemma indep}, $S^{(i)} =^d S  -X_i+ X_i^*$ so that 
\eqref{pre mixture}
simplifies  to:
for bounded nonnegative measurable $g$,
\begin{equation}\label{this identity}
   \e g(S^*)= \sum   \frac{a_i}{a}\  \e g(  S -X_i + X_i^*   ).
\end{equation}
The result above says precisely
that $S^*$ can be represented by the mixture of the distributions of
$S + X_i^* -X_i $ with mixture probabilities $a_i/a$. 
With a random $I$ having distribution  defined by 
\begin{equation}\label{wanted by larry}
\p(I=i)=a_i/a, 
\end{equation}
and  all of $I,(X_1,X_1^*),(X_2,X_2^*),\ldots$ mutually 
independent, the mixture formula \eqref{this identity}  can be restated as
\begin{equation}\label{mixture}
   S^* =^d  S -X_I+ X_I^* .
\end{equation}
In the preceding coupling,  for each $i$,  marginal distributions of $X_i,X_i^*$ are specified, but the joint distribution of
$(X_i,X_i^*)$ is otherwise arbitrary.  Allowing such dependence is important  for use with Stein's method; see Section \ref{sect chen}.
Of course,   mutual independence for 
$I,X_1,X_2,\ldots,$ $X_1^*,X_2^*,\ldots$ implies mutual independence for  
$I,(X_1,X_1^*),(X_2,X_2^*),\ldots$  .

For each case, $S=X_1+\cdots+X_n$ or $S=X_1+X_2+\cdots$,  \eqref{mixture} can be written out with notation to emphasize that a single term has been biased\footnote{and hence the title of this paper}:
\begin{equation}
\label{noniid}
(X_1+X_2+\cdots+X_n)^*=^dX_1+ \cdots + X_{I-1}+X_I^*+X_{I+1}+\cdots+X_n,
\end{equation}
and
\begin{equation}
\label{noniid infinite}
(X_1+X_2+\cdots)^*=^dX_1+ \cdots + X_{I-1}+X_I^*+X_{I+1}+\cdots.
\end{equation}
It is a natural abuse of notation 
 to view \eqref{noniid} as a special case of \eqref{noniid infinite}.  The reason that this is abuse, rather than the special case $X_{n+1} = X_{n+2}=\cdots = 0$ is that  the identically zero random variable $X$ cannot be
size biased.  Specifically, $X=0$ doesn't satisfy the conditions of the definition in 
\eqref{size bias RN}, and size biasing this $X$, if allowed,  would abrogate 
Lemma \ref{lemma biject}.  Nonetheless, it is customary to  follow the notational abuse that if $X=0$ then $X^* =^d X = 0$,
so that one can view \eqref{noniid} as the
special case of \eqref{noniid infinite}, and later, write formulas
such as \eqref{poisson sum} for a sum with infinitely many terms,
without writing out a second instance for a sum with finitely many terms.

In contrast to a sum of independent  nonnegative summands, which is size biased by biasing a \emph{single}  term, 
a product $W= X_1 X_2 \cdots X_n$, of  independent, nonnegative random variables
$X_1, \ldots, X_n$,  each with finite, strictly positive mean, is size biased by biasing \emph{every} factor:  taking
$X_1^*,\ldots,X_n^*$ independent, one has 
\begin{equation}\label{bias a product}
W^* \indist X_1^*\cdots X_n^*.
\end{equation}
Here, we leave the proof as an exercise; this result comes from \cite{minluk}.
For the case of dependent summands, the decomposition 
\eqref{pre mixture}  is useful; in contrast, for dependent factors, we
don't know of any useful relation.

An interesting  example of the use of \eqref{noniid infinite} involves 
$S=\sum_{i \ge 1} 2 B_i/3^i$ with  independent $B_i$, 
with 
$\p(B_i=0)=\p(B_i=1)=1/2$.  The cumulative distribution of this sum $S$ is known as the Cantor function;  the distribution of $S$ is, by all reasonable interpretations, the uniform distribution on the Cantor middle thirds set.  By \eqref{wanted by larry}, the random
index $I$ has the geometric distribution $\p(I=i)=2/3^i$ for
$i=1,2,\ldots$, and by \eqref{bernoulli},  the size biased version of
$B_i$ is $B_i^* =  1 = B_i+(1-B_i)$,
so that \eqref{mixture} simplifies to 
$$  
S^* =^d S + 2(1-B_I)/3^I.
$$

A closely related example, using the same $B_i$, is the standard
uniform (0,1) random variable $U  =\sum_{i \ge 1}  B_i/2^i$. With a
random index $J$ having  geometric distribution $\p(J=i)=1/2^i$ for $i=1,2,\ldots$,
independent of $B_1,B_2,\ldots$, \eqref{mixture} simplifies to 
\begin{equation}\label{uniform toy}
U^* =^d U + (1-B_J)/2^J=
\frac{B_1}{2}+\frac{B_2}{4}+\cdots+\frac{B_{J-1}}{2^{J-1}}+\frac{1}{2^J}
+\frac{B_{J+1}}{2^{J+1}}+ \cdots.
\end{equation}
Of course,  it is easy to calculate that the density of $U^*$ is $2x$ on (0,1), using \eqref{sizebias-f}:  multiply the density of the uniform by $x$ and divide by $\e U = 1/2$.   But perhaps the following exercise is not easy.

{\bf Exercise}  Prove, without using size bias, that the sum on the
right side of \eqref{uniform toy} has density $f(x)=2x$ on (0,1).

\begin{figure}[ht]

\includegraphics[scale=.4]{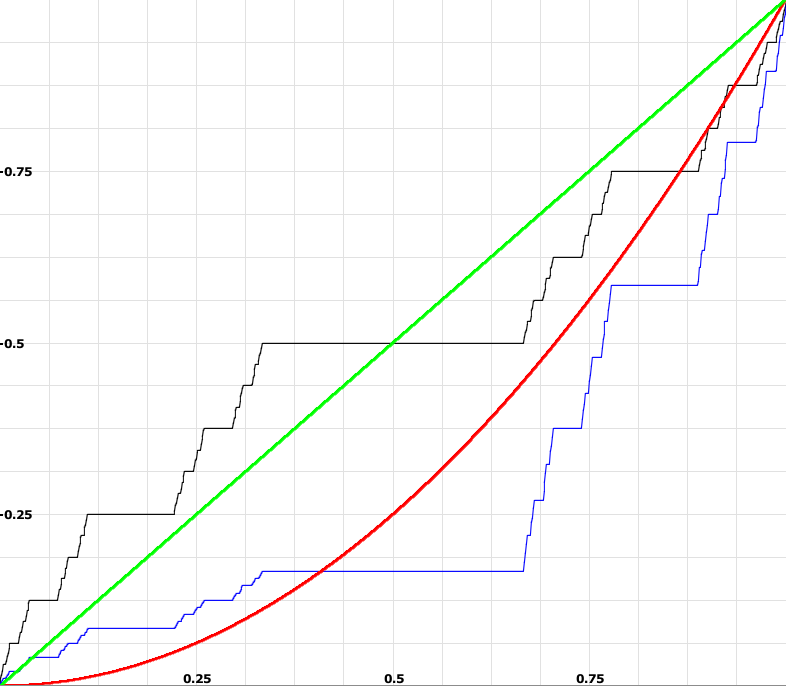} 
\caption{Cumulative distribution functions for the uniform distribution on (0,1), the uniform distribution on the Cantor set, and the size biased versions of these.  Image produced using MathStudio \cite{studio}.\label{figure cantor}}
\end{figure}

For the case with a finite number of summands, where the summands are not only independent
but also {\em identically distributed}, 
the recipe \eqref{noniid} simplifies.
In this case it does not matter which summand is biased, as all the distributions in the
mixture are the same; hence we may replace the random $I$ with the fixed $i=1$, yielding
\begin{equation}\label{iid}
(X_1+X_2+\cdots+X_n)^* \indist X_1^*+ X_{2}+X_3+\cdots+X_n.
\end{equation}

Here are some elementary consequences of \eqref{iid}. 
 Recall \eqref{bernoulli}, that for $p \in (0,1]$, a Bernoulli random variable with mean $p$, size biased, is the constant 1. 
Summing $n$ independent copies gives us random
variables $S_n$ whose distribution is Binomial($n,p)$.  Hence using
\eqref{iid},
\begin{equation}\label{binomial n p}
  S_n^* =^d 1+ S_{n-1}.
\end{equation}
Finally, taking $\lambda \in (0,\infty)$ fixed, $Z$ to be
Poisson($\lambda$), and $X_n$ to be Binomial($n,\lambda/n)$, the
Poisson limit for the Binomial, together with Theorem 
\ref{limit theorem} and \eqref{binomial n p}, implies 
$Z^* \indist Z+1$.
Of course, this equality was already verified by direct calculation using
\eqref{sizebias-f} and \eqref{baby poisson},
but the beauty of the argument via \eqref{binomial n p} is that it is
purely conceptual.

\subsubsection{Example:  compound Poisson} 
\ignore{
The formula \eqref{add one} for size biasing a single Poisson distributed random
variable, together with the scaling property \eqref{scaling}, combine
nicely with the formula \eqref{noniid infinite} for size biasing a sum of
independent, non identically distributed summands.   
}
Given the distribution for a discrete positive random variable $Y$ with finite mean,  and $0 < a < \infty$, we will show how to construct
a distribution for $S$ such that 
\begin{equation}\label{goal first}
   S^* =^d S+Y \text{ with } S,Y \text{ independent, and } \e S = a.
\end{equation}   
To specify the distribution of $Y$, suppose that $p_i = \p(Y=y_i)$, for distinct constants 
$y_1,y_2,\ldots >0$,  with $p_1+p_2+\cdots =1$.  The requirement  $\e Y
< \infty$ becomes \mbox{$ \sum p_i y_i < \infty$}.  
Define $$\lambda_i = a \,p_i/y_i.$$  Let  $Z_i$ be
Poisson with mean $\lambda_i$ with $Z_1,Z_2,\ldots$ mutually independent.   We will show that
\begin{equation}\label{poisson sum}
S=\sum_{i \ge 1} X_i, \ \ \  \mbox{ with }X_i=y_i Z_i
\end{equation}
gives a solution to \eqref{goal first}, using only formula \eqref{add one} for size biasing a single Poisson distributed random
variable, the scaling property \eqref{scaling}, formula \eqref{noniid infinite} for size biasing a sum of
independent, non identically distributed summands, and the trivial calculation that $a_i := \e X_i = \lambda_i \, y_i = a p_i$, hence $\e S =
\sum \lambda_i \, y_i =  \sum  a p_i  = a$.

First, using \eqref{add one}, $Z_i^* =^d Z_i +1$.  Second, using the scaling property \eqref{scaling}, 
$X_i^*=^d X_i +y_i$.  In the recipe \eqref{noniid infinite}, there is a random index $I$, independent of the $X_1,X_2,\ldots$, with 
\begin{equation}\label{compound dist I}
\p(I=i) =\e X_i / a = \lambda_i y_i /a=p_i,  
\end{equation}
and we can take the coupling in which $X_i^* = X_i + y_i$ for each $i$.  This yields
$S^*=^dS+y_I$, with $S,I$ independent.   Since the $y_i$ are distinct, for each $i$, as events, $(Y_I=y_i) = (I=i)$, hence the distribution of $y_I$ \emph{is} the given distribution for $Y$.    To summarize, we were given the distribution for $Y$, and we constructed a distribution for $S$ so that \eqref{goal first} holds.    We will revisit 
the relation $S^* =^d S +Y$ with $S,Y$ independent
in Section \ref{sect inf div}; the preceding is
then seen as an explicit example of \eqref{indep}, with the distribution of $Y$ specified in advance.  In the standard
literature, the random variable $S$ in \eqref{poisson sum} is said to
have a \emph{compound Poisson} distribution, given the further
restriction that $\sum_i  \lambda_i <\infty$. Compound Poisson with
finite mean requires \emph{both}  $\sum_i  \lambda_i <\infty$ and   $\sum
\lambda_i \, y_i <\infty$;
in contrast, we require \emph{only} the latter.

Recall, if $Z$ is Poisson($\lambda)$ then its probability generating
function is $G_Z(s) := \e s^Z = \exp(\lambda(s-1))$. Substituting
$s=e^\beta$, the
moment generating
function of $Z$ is $M_Z(\beta) := \e e^{\beta \, Z} =  \exp(\lambda(e^\beta-1))$.
Hence in \eqref{poisson sum}, the moment generating function of $X_i$
is   $M_{X_i}(\beta) = $ \mbox{$ \exp(\lambda_i(e^{\beta \, y_i}-1))$} and the moment
generating function of $S$ is
\begin{equation}\label{comp pois M}
   M_S(\beta) 
 = \exp \left( \sum_i
     \lambda_i(e^{\beta \, y_i}-1) \right) =
\exp \left(a \sum_i
     \frac{e^{\beta \, y_i}-1}{y_i} \, \p(I=i) \right),
\end{equation}
with the distribution of $I$ given by \eqref{compound dist I}.  Likewise, the characteristic function of $S$,
$\phi_S(u) := \e e^{i u S}$ is given by
\begin{equation}\label{comp pois phi}
   \phi_S(u) 
 = \exp \left( \sum_k
     \lambda_k(e^{i u  \, y_k}-1) \right) =
\exp \left(a \sum_k
     \frac{e^{i u \, y_k}-1}{y_k} \, \p(I=k) \right).
\end{equation}

\section{Waiting time paradox: the renewal theory  connection}\label{sec:renewal}
We resolve the waiting time paradox from Section \ref{sec:waiting} in the general context of renewal processes, at the same time 
providing a conceptual explanation of the identities \eqref{noniid}
and 
(\ref{iid}). 

Let the interarrival times in Section \ref{sec:waiting} be denoted $X_i$
so that, starting from $0$, arrivals occur at
times $X_1, X_1+X_2, X_1+X_2+X_3,\ldots$,
and assume only that the $X_i$ are
i.i.d., strictly positive random variables with
finite mean; the  paradox presented earlier was for the special case with $X_i$ exponentially 
distributed. 

The following argument is heuristic.  
One way to model the ``arbitrary instant $t$'' is to choose a random 
$T$ uniformly from 0 to $l$, 
independent of $X_1,X_2,\ldots$, and then take
the limit as $l \rightarrow \infty.$ For large but finite $l$, conditional on $X_1,X_2,\ldots$,
apart from possible cutoff at the extreme right\footnote{Conditional
  on $T=t$ and $X_1+\cdots+X_{m-1} < t < X_1 + \cdots + X_{m-1}+X_m$, 
there are $m$ interarrival intervals,  and for $i=1$ to $m-1$ interval
$i$ is selected with probability proportional to $X_i$, 
but interval $m$ is selected with probability proportional to $t-(X_1+\cdots+X_{m-1}) < X_m$.}  the probability of $T$ landing in a given 
interarrival interval is 
 proportional to its length. 
In other words, if the interarrival times $X_i$ have a distribution $dF(x)$, 
the distribution of the length of the selected interval is approximately
proportional to $x \ dF(x)$. In the limit, it is precisely correct that the
distribution of the length of the selected interval is the
distribution of $X^*$. 

For the particular
 case of exponentially distributed interarrival times, the density of $X^*$ is $xe^{-x}$, with mean value $2$, and so a 
 right--left 
 symmetry argument
gives the answer in a).

A conceptual explanation of identity (\ref{iid}) is given by the following heuristic. 
Group the interarrival intervals into successive blocks of $n$ intervals.
By considering only the endpoints of blocks, i.e., 
the renewal process,  decimated by $n$, 
 the random time $T$ must find
itself in a block with total length distributed as $S^* = (X_1+\cdots+X_n)^*$.
But regardless of the grouping, the random time $T$ still finds itself in an internal interval whose length is distributed as the size biased distribution of the interarrival times; the lengths of the other intervals in the same block are not affected. Thus the total block length must also 
be distributed as $X_1+\cdots+X_{i-1}+X_i^*+X_{i+1}+\cdots+X_n$.  
A small extension of this heuristic may convince one of the identity \eqref{noniid}:  given $n$ distributions for strictly positive $X_1,\ldots,X_n$ all with finite mean, create the  $n$-alternating renewal process, in which the independent interarrival time distributions cycle through the $n$ given distributions.   The decimation by $n$  has independent interarrival times distributed as $ S= X_1+\cdots+X_n$, with independent summands, $T$ picks out a block with length distributed as $S^*$, and  the contribution $\e X_i$ makes to the total block size  governs the distribution of which subinterval in a block gets chosen by $T$.    And for \eqref{noniid infinite}, where $S=X_1+X_2+\cdots$ with $\e S < \infty$,  another small extension of the heuristic may be convincing.  But we don't really expect the $\infty$-alternating renewal process to become a popular model.



The standard rigorous analysis of the waiting time paradox, for instance in \cite{thorisson}, is a bit less direct, based on randomizing the starting point 
of the arrivals, 
so that the arrival times form
a stationary sequence.  Begin by extending $X_1,X_2,\ldots$ to an independent, 
identically distributed sequence \mbox{$\ldots,X_{-2},X_{-1},X_0,X_1,X_2,\ldots$ \ .} Informally, 
if the arbitrary instant $t$ could be uniform on the whole line (or by adapting the above 
limiting argument) 
then 
$t$ would fall uniformly inside a size biased interarrival interval; relabeling, we call $t$ by the name zero, and the landing interval
has length $X_0^*$.
Then the prior arrival and next 
arrival would be at times $-(1-U)X_0^*$ and $UX_0^*$ respectively, where the uniform 
$U \in [0,1]$ is independent of the $X_i$'s. Thus motivated, we define a process by setting 
arrivals at positive times $UX_0^*, UX_0^*+X_1,UX_0^*+X_1+X_2,\ldots$, as well as negative times
$-(1-U)X_0^*,-((1-U)X_0^*+X_{-1}),-((1-U)X_0^*+X_{-1}+X_{-2}),\ldots$\ .
It can be proved that this process is stationary, see
\cite[Theorem 8.1, Chapter 8]{thorisson}. Our desired waiting time 
$W_t$ is then equal in distribution to  $W_0 = UX_0^*$.

The interval which covers the origin has expected length
$\e X_0^* = \e X_0^2 / \e X_0$ (by (\ref{moment shift}) with $n=1$,)
and the ratio of this to $\e X_0$ is $\e X_0^*/\e X_0 = \e X_0^2 / (\e X_0)^2$.
By Cauchy-Schwarz, this ratio is at least 1, (see also \eqref{lies above},) and every value in $[1,\infty]$ is feasible.
Since the mean waiting time is $\e W_t = \e W_0 = \e (U X_0^*) = (1/2) \e X_0^*$,
the ratio  $\e W_t/ \e X_0$ can be any value
between 1/2 and infinity, depending on the distribution of $X_0$.

The exponential case is very special, where 
``coincidences''
effectively hide all the structure involved in size biasing.
As suggested by Feller's argument  (a) at the start of this paper, but now
justified by stationarity, 
 $\e W_t = 1$.
Furthermore, for the exponential case, where $X_0$ has density $e^{-x}$ for $x>0$, one gets $X_0^*$ has density $x e^{-x}$ and  the two summands $UX_0^*$ and
$(1-U)X_0^*$ are independent, 
each with the original exponential 
distribution.\footnote{Exercise for the reader: prove that if $U \, X^* =^d X$ when $U$ is independent of $X^*$ and $U$ is distributed uniformly on $(0,1)$, then $X$ has an exponential distribution  --- on some scale. Not hard;  or, see \cite{pakes}.} Thus the general recipe for cooking up a stationary
process, involving $X_0^*$ and $U$ in general, simplifies beyond recognition:
the original simple process with arrivals at times $X_1, X_1+X_2,
X_1+X_2+X_3,\ldots$ forms half of a stationary process,
which is completed by its other half, arrivals at $-X_1', -(X_1'+X_2'), \ldots, $
with $X_1,X_2,\ldots,X_1',X_2',\ldots$ all independent and
exponentially distributed.

The above material deals with \emph{renewal} processes, and perhaps
originated in Doob  \cite{Doob}.   A broad generalization, applying to
\emph{stationary} point processes --- dropping the requirement that
the interarrival times be independent --- was given by
\cite{Kaplan}.   See also \cite[p. 299]{Daley}.

\section{Size bias in statistics}


We now touch briefly on the topic of inadvertent or unavoidable size bias\footnote{or length bias, as it is sometimes
called in sampling literature} in statistical sampling by citing two references from a vast literature.\footnote{
An unpublished survey by Termeh Shafie on Length-Biased Sampling,
    found on her ETH webpage, contains a quite useful bibliography.}
We also  discuss the \emph{deliberate} use of size bias, as a sampling tool.


\subsection{Inadvertent size bias}
In a 
 1969 paper \cite{COX} David Cox identifies, among other topics, length bias in a 
then-standard process
for estimating the mean length of textile fibers:  In outline, as he describes it, 
fibers are 
 gripped 
by a pincer, all ungripped fibers adhering to the gripped ones are carefully removed, and the 
remaining fibers are measured. Cox points out that since shorter
fibers are more likely to be missed by the pincer, the distribution of the sampled lengths is
length biased. He proposes some adapted estimators for getting at parameters of the original 
distribution if the sampling process itself cannot be refined.

Nearer to the present, the 2009 paper \cite{KP} considers issues
arising in assessing the value of medical screening and the 
effects of subsequent early treatment on survival time. As discussed
in \cite{KP}, for reasons
analogous to waiting-time bias, the durations 
of preclinical disease states detected by certain screening protocols are subject to length bias.
Even though the durations themselves are not observed, longer durations are likely to derive from 
slower-acting instances of the disease under consideration, and hence are correlated a priori with 
longer survival times. Therefore, as indicated by the authors, 
improvement in survival time is likely to be overestimated by 
such studies  if suitable adjustments are not made.

\subsection{Deliberate size bias to create something unbiased}
Somewhat paradoxically, size biasing can occasionally be used to construct
{\em unbiased} estimators of quantities that would seem, at first glance, difficult to
estimate without bias. The following procedure for unbiased ratio estimation is due to Midzuno 
\cite{midzuno}; see also Cochran  \cite{C77}. Suppose that for each individual $i$ in some large 
population there is a pair of numbers $(x_i,y_i)$, with 
the value $x_i$ easy to obtain but $y_i$ more difficult.  Assume each
$x_i \ge 0$, with not all zero.  Suppose that it is desired to
estimate the ratio $\sum_i y_i / \sum_i x_i$ without bias and without
sampling  the entire population.
Perhaps $x_i$
is how much the $i^{th}$
customer was billed by their utility company last month, and $y_i$, say a smaller
value than $x_i$, the amount they were supposed to have been billed.
Suppose we would like to know just how severe the overbilling error
is; that is, we would like
to know the
`adjustment factor',  the ratio $\sum_i y_i / \sum_i x_i$.
Even though $ \sum_i x_i$ is known, collecting
the paired values for everyone is laborious and expensive, so we would like to be able to
use a sample of $m<n$ pairs to make an estimate. 
It is not hard to verify that, 
if we select a set $R$ of $m$ indices, with all ${n \choose m}$ sets
equally likely, then 
the estimate $\sum_{j \in R}y_j/\sum_{j \in R}x_j$ will be biased.

 The following device gets around this difficulty. Draw a random set $R$ of size $m$ by first 
selecting  $i$ with  size-biased probability $x_i/\sum_j x_j$. Then draw $m-1$ indices uniformly from the 
remaining $n-1$. Though we are out of the independent framework, the principle of (\ref{iid}) is 
still at work: size biasing one element has size biased the sum. 
This is so because we have size biased the
one, and then chosen the others from  the 
appropriate conditional distribution. Thus, we have  
selected a set $r$ of indices with probability proportional to $\sum_{j \in r}x_j$. 
>From this observation it follows that
$\e(\sum_{j \in R}y_j/\sum_{j \in R }x_j)=\sum_j y_j / \sum_j x_j$. 

Here is Midzuno's procedure in a bit more detail.  
Let
$$
 \overline{x}=\frac{1}{n}\sum_{j=1}^n x_j \text{ and } \overline{y}=\frac{1}{n}\sum_{j=1}^n y_j .
$$
First choose index  $I$ with distribution
$$
P(I=i)=\frac{x_i}{\sum_{j=1}^n x_j}.
$$
Then from the remaining set $\{1,\ldots,n\} \setminus \{I\}$,  take a simple
random sample $ S$ of size $m-1$. Let $R=S \cup \{I\}$ be the resulting set of size
$m$. We claim the estimator $T_{R}$ is unbiased for $\overline{y}/\overline{x}$, 
where, 
for $r$ a subset of $\{1,\ldots,n\}$,
$$
T_r=\frac{\overline{y}_r}{\overline{x}_r} \quad \mbox{with} \quad
\overline{y}_r=\frac{1}{m}\sum_{j \in r} y_j \quad \mbox{and}
\quad \overline{x}_r=\frac{1}{m}\sum_{j \in r} x_j.
$$
To see why, consider that $R$ may equal $r$, any set of size $m$,
in $m$ different possible ways, one each according to
first selecting some element $i \in r$ with probability $P(I=i)$, and
then collecting the remaining elements in the simple random sample. 
Hence, \begin{eqnarray*}
P( R=r)&=&\sum_{i \in r}P(I=i)P( S\setminus \{i\}=r\setminus \{i\}) \\
&=& \sum_{i \in r} \frac{x_i}{\sum_{j=1}^n x_j} \frac{1}{{n-1
\choose m-1}}\\
&=& \frac{1}{{n-1 \choose m-1}} \frac{\sum_{i \in r} x_i}{\sum_{j=1}^n x_j}\\
&=& \frac{1}{\frac{n}{m}{n-1 \choose m-1}} \frac{\overline{x}_r}{\overline{x}}\\
&=& {n \choose m}^{-1} \frac{\overline{x}_r}{\overline{x}}. \end{eqnarray*}
Next, applying the easily shown identity
\begin{eqnarray*}
{n \choose m}^{-1} \sum_{|r|=m} \overline{y}_r =\overline{y},
\end{eqnarray*}
we obtain
\begin{eqnarray*}
ET_{R}&=&\sum_{|r|=m} \frac{\overline{y}_r}{\overline{x}_r}P( R=r) = {n \choose m}^{-1} \sum_{|r|=m} \frac{\overline{y}_r}{\overline{x}_r}\frac{\overline{x}_r}{\overline{x}}\\
&=& \frac{1}{\overline{x}}{n \choose m}^{-1} \sum_{|r|=m}
\overline{y}_r\\
&=& \frac{\overline{y}}{\overline{x}}.
\end{eqnarray*}
For the variance of the estimator, see \cite{rao}.

\section{Relation to Stein's method and concentration inequalities} \label{sect chen}

Implicit in Chen 1975 \cite{chen},  with improved constants due to
\cite{BHJ92}, see also \cite[Theorem 4.12.12]{grimmett}, is the following result from \cite{gr06},
Theorem 1.1, see also \cite[Theorem 4.10]{ross1},
which we paraphrase\footnote{The theorem in \cite{gr06} is
stated with the condition that $X$ be a finite sum of indicator random
variables.  However, an arbitrary nonnegative integer valued $X$ is a
sum of indicators, namely $X = \sum_{i \ge 1} 1(X \ge i)$, and the
restriction on finite sum can be removed using 
Theorem \ref{limit theorem} applied to $X_n := X \wedge n =  \sum_{i = 1}^n 1(X \ge i)$.}
here as
\begin{theorem}\label{chen theorem}
Let $X$ be a nonnegative integer valued random variable with $\lambda
:= \e X \in (0,\infty)$;  let $Z$ be Poisson with parameter
$\lambda$. Then for any coupling of $X$ with $X^*$, the total variation distance between
the distributions of $X$ and $Z$
satisfies
$$
   d_{\rm TV}(X,Z) \le (1-e^{-\lambda}) \ \e | X^* - (X+1)|.
$$
\end{theorem}

The total variation distance appearing in Theorem \ref{chen theorem}
is defined, for random variables $X,Y$ in general, by $   d_{\rm TV}(X,Y)
= \sup_B (\p(X \in B) - \p(Y \in B) )$, with the supremum taken over
all Borel sets.

Size biasing also has a connection with Stein's method for obtaining error bounds when
approximating 
distributions 
by the normal distribution, see 
\cite{BR89,BRS89,goldstein-book,GR96}.

Size bias also plays a role in concentration inequalities, see \cite{ghosh,ghosh2,arratiagamma,BGI}. The results from \cite{ghosh,arratiagamma} include:
if $X \ge 0$ with $a :=\e X \in (0,\infty)$ can be coupled to $X^*$ so that $\p(X^* \le X+c)=1$,
then
\begin{eqnarray*}
\text{for }0<x \le a,  \p(X \le x) & \le &(a/x)^{x/c} e^{(x-a)/c} \le \exp(-(a-x)^2/(2ca)), \\   
\text{for }x \ge a, \p(X \ge x) & \le &(a/x)^{x/c} e^{(x-a)/c} \le  \exp(-(x-a)^2/(c(a+x))).
\end{eqnarray*}

To see how size bias enters, if a coupling satisfies  $\p(X^* \le
X+c)=1$, then for all $x$, the event $X^* \ge x$ is a subset of the event $X \ge x-c$.
 Hence for $x>0$,
\begin{eqnarray*}
x \p(X \ge x)  = x \, \e 1(X \ge x) & \le & \e(X 1(X \ge x)) \\
 &= & a \, \p(X^* \ge x)
 \\
 &\le& a \,\p(X \ge x-c),
\end{eqnarray*}
and dividing by $x$ we get
\begin{equation}\label{upper one step}
\forall x>0, \ \      G(x) \le \frac{a}{x} \, G(x-c),
\end{equation}
Iterating \eqref{upper one step} leads to the sharp upper bounds on $\p(X \ge x)$, for each
$x \ge a$.  An extension to exploit the weaker condition
$\p(X^* \le  X+c | X^*) \ge p \in (0,1)$  is discussed in \cite{cook}.

In the context of sums of independent random variables each with a
bounded range,  the concentration bounds based on bounded size bias
couplings are stronger than the corresponding Chernoff--Hoeffding
bounds, as well as being broader in scope;  see \cite{arratiagamma}. 
 Applications of these bounds to
situations involving dependence, such as  the number of relatively ordered
subsequences of a random permutation, 
sliding window statistics including the number of $m$-runs in a sequence
of coin tosses, 
the number of local maxima of a random function on a lattice, the
number of urns containing 
exactly one ball in an urn allocation model, 
and the volume covered by the union of $n$ balls placed uniformly over
a volume $n$ subset of $\BR^d$, are discussed in \cite{ghosh2}.
An example showing that the size bias concentration bounds  supply a
desired uniform integrability, in a situation where the usual
Azuma-Hoeffding  bounded martingale difference inequality is not
adequate, is given in \cite{agk}.


\section{Size bias and Palm distributions}


The size bias view of arrival times and stationarity, discussed in Section
\ref{sec:renewal}, is sometimes
expressed in the language of Palm measures for stationary point processes;
see \cite[Chapter 8]{thorisson} or \cite[p. 299]{Daley} for details.  At this level,
Palm measures are derived from \emph{simple} point processes, that is, random
nonnegative integer valued measures $\xi$ for which any singleton set
$\{s\}$ has
measure zero or one,  and the Palm measure $\xi_s$ corresponds to conditioning on
having an arrival at the point $s$.

There is a more general version of Palm measure, which  applies to
nonnegative random measures; we attribute this to
Jagers and Kallenberg, \cite{jagers,kallenberg73,kallenberg75}.  This version
is, quite directly, a generalization of biasing 
a process $\BX  
= (X_1,X_2,\ldots) \in [0,\infty)^{\mathbb{N}}$ in the direction of its
$i$th coordinate, to get ${\bf X}^{(i)}$, described in Section
\ref{sect basics one}.
The setup is:   $S$ is a complete separable metric
space and  $M$ is the set of nonnegative sigma-finite measures on $S$;
typical examples include $S=\mathbb{R}$ and $S = \mathbb{R}^d$.
Fix a random measure $\xi$, that is, a random element of $M$.
The characterizing property of the Palm measures $\xi_s$, for $s \in
S$,
 is that, for bounded
measurable functions $g\! : M \to \mathbb{R}$, 
\begin{equation}\label{palm bias}
  \e g( \xi_s) = \frac{\e (\xi(ds) g(\xi) )}{\e \xi(ds)}.
\end{equation}

In the restrictive case $S=\mathbb{N}$, a measure $\zeta \in M$
corresponds naturally to the sequence $(z_1,z_2,\ldots) \in
[0,\infty)^{\mathbb{N}}$ with $z_i = \zeta(\{i\})$, the mass assigned by the
measure to the location $i$ in the underlying space $S$,  hence a
random measure $\xi$ corresponds to a stochastic process $\BX=(X_1,X_2,\ldots)$
with values in $[0,\infty)^{\mathbb{N}}$.  In this restrictive case and under
this correspondence, with $s
= i$,  $\xi_s = \BX^{(i)}$ is the process $\BX$ biased by its $i$th
coordinate $X_i$, and $\e \xi(ds) = \e X_i =: a_i$, and 
\eqref{palm bias} looks identical to \eqref{sizebias infinite}  ---
the only difference is that in  the setup for \eqref{sizebias infinite}
we needed to \emph{assume} that for each $i$, $\e X_i >0$  --- in
particular, one cannot size bias the random variable $X$ which is
identically zero. But in the measure context, it would be an
unreasonable extra assumption, to require that the intensity measure
$\e \xi$ be purely atomic.

The above has fully described a sense in which Palm measures are a
generalization of simple size bias.
As an application, we provide a solution, in the same spirit, 
to (part of) Exercise 11.1 in  \cite{kallenberg75}.  The 
exercise asks for a proof of the following theorem, in which the
emphasis 
 is that $\xi$ is
\emph{not}
assumed to be  \emph{integer-valued}.

\begin{theorem}\label{thm palm poisson}
Suppose  $\xi$ is a random measure on $S$.  For $s \in S$ write 
$\delta_s$ for deterministic measure ``unit mass at $s$''.  Suppose
the
 Palm measures satisfy:  for  $s \in S$, $\xi_s=\xi + \delta_s$.  Then
$\xi$ is a Poisson process.  
\end{theorem}

\begin{lemma}\label{lemma new1}
Under the hypotheses of Theorem \ref{thm palm poisson}, for any
measurable  $B \subset S$ for which $\e \xi(B) \in (0,\infty)$, the
random variable $X = \xi(B)$ satisfies  $X^* =^d X+1$.
\end{lemma}
\begin{proof}
Write $\mu = \e \xi$ for the intensity measure;  this is the
deterministic element of $M$ with $\mu(B) = \e ( \xi(B))$ for
measurable $B \subset S$.  With this notation, \eqref{palm bias} says
\begin{equation}\label{palm bias mu}
  \e g( \xi_s) = \frac{\e (\xi(ds) g(\xi) )}{\mu (ds)}.
\end{equation}

The characterization of Palm measures, as written in
 \eqref{palm bias mu}, is shorthand for its multiplied out version,
$$ \e g( \xi_s)\  \mu(ds)= \e (\xi(ds) g(\xi) ),
$$
so that for measurable $B \subset S$, 
$$
 \int_B  \e g( \xi_s)\  \mu(ds)
= \int_B  \e (\xi(ds) g(\xi) ) = \e \int_B  (\xi(ds) g(\xi) ) = \e
(g(\xi) \xi(B)).
$$
Now fix a measurable $B \subset S$ with $a := \e \xi(B) \in (0,\infty)$,
and fix a bounded measurable $f: \mathbb{R} \to \mathbb{R}$. This
induces
a bounded measurable function  $g\! : M \to \mathbb{R}$ 
via $g(\zeta) = f(\zeta(B))$.  This yields
$g(\xi)=f(\xi(B))$,
so with $X = \xi(B)$ the right side of the display above is $\e (f(X)
X)$.  From the hypothesis $\xi_s=\xi + \delta_s$ we have, for every $s
\in B$, $g(\xi_s) = f(\xi_s(B))=f((\xi +
\delta_s)(B))=f(\xi(B)+1)=f(X+1)$, and the left side of the display is
$ \int_B \e g( \xi_s)\  \mu(ds)= \e   f(X+1) \int_B  \mu(ds)=\e f(X+1) \,
\mu(B) = \e f(X+1)\,  \e X$.
Hence, for bounded measurable $f$, $\e(X f(X)) = \e f(X+1)  \, \e X$.
This last relation, now proved for an arbitrary bounded measurable 
 $f: \mathbb{R} \to \mathbb{R}$, shows that $X^* =^d X+1$.
\end{proof}

\ignore{   IGNORE
Remark:  Lemma \ref{lemma new1}, combined with Corollary \ref{cor poisson}, shows that $X$ is
Poisson
with parameter $a= \e \xi(B)$, and this is one of the two
characterizing properties of Poisson processes.  The other property is
mutual independence:  for disjoint measurable sets $B_1,B_2,\ldots
\subset S$, the random variables $X_1:= \xi(B_1), X_2 := \xi(B_2),...$
are
mutually independent. 
}    

\begin{lemma}\label{lemma new2}
Suppose that $G$ is an event, and $X$ is Poisson($\lambda$).  Write
\begin{eqnarray*}
  p_i &=& \p(X=i),\\
q_i &=& \e(1(X=i) 1(G)),\\
r_i &=& \e(1(X=i) 1(G^c)),
\end{eqnarray*}
so that $p_i=q_i+r_i$ for $i=0,1,2,\ldots$.
Suppose that for $i=0,1,2,\ldots$,
\begin{equation}\label{cite recursions}
 (i+1)\, q_{i+1}= \lambda \,q_{i}, \ \ \ \ (i+1)\, r_{i+1}= \lambda \, r_{i}.
\end{equation}
Then $X$ and $G$ are independent.
\end{lemma}
\begin{proof}
The proof is similar to that of Proposition \ref{proposition baby}.
 In particular, applying induction as there to \eqref{cite recursions}, we obtain $q_i=\p(G)e^{-\lambda}\lambda^i/i! = \p(G)\p(X=i)$.
\end{proof}
\begin{proof}[Proof of Theorem \ref{thm palm poisson}]
Consider a measurable $B \subset S$ for which $\lambda :=  \e
\xi(B) \in (0,\infty)$.  Lemma \ref{lemma new1}, combined with Corollary \ref{cor poisson}, shows that 
$X=\xi(B)$ is
Poisson($\lambda$).  Now consider an event  $G$ which is
measurable with respect to the restriction of $\xi$ to $B^c$.  Lemma
\ref{lemma new2} shows that $G$ is independent of $X$, with
an argument similar to the proof of
Lemma \ref{lemma new1} verifying that the hypotheses of Lemma
\ref{lemma new2} are satisfied.   Hence for disjoint subsets  $B_1,B_2,\ldots
\subset S$ each having $0 < \e \xi(B_i) < \infty$, $\xi(B_1)$ is
independent of $(\xi(B_2),\xi(B_3),\ldots)$, so by induction, the
Poisson distributed random variables $\xi(B_2),\xi(B_3),\ldots$ are
mutually independent.
\end{proof} 

\emph{Acknowledgement}.   We are grateful to the referee who asked us
to consider the the connection between size bias and Palm measures.

\section{Martingale size bias, and size bias for Galton Watson trees}\label{sect bias GW}

This section is based mostly on  \cite{LPP}, which employs a notion of size-biased Galton Watson trees to 
give a  
conceptual and intuitive proof of the Kesten-Stigum theorem, which we briefly describe below.
We also found  \cite{ZhanShi}, \cite{Dawson}, and \cite{LyonsPeres}  useful for clarifying 
lingering issues involving  
the ``spine'' or ``backbone'' of the size biased Galton Watson trees. 

For the reader 
already familar with the size biased Galton Watson tree, here is a brief description of these issues.
a) Is the spine intrinsic to the size-biased tree, or is it just an ingredient in a 
particular construction? b) Given just the tree, generated using a spine but without labels
to show where the spine lies, can the spine be located? 
c) Can one start with the unbiased Galton Watson tree, and then add a process, of immigrants and their descendants, to get a coupling with a size biased tree?  d)  If yes to c), can this be done so that the original tree and the difference are independent?  We answer a) and b), but leave c) and d) 
alone.\footnote{Our reason for leaving c) and d) alone is that there are conflicting notions of the use of immigrants
in constructing the size biased tree (though leading, in the end, to the same   distribution on trees).  
In \cite{LPP}, the notion is implicitly established by declaring that the size biased process, with spine removed, is a branching process with immigration --- starting with zero individuals, so that every individual is either an immigrant, or else descended from an immigrant.  Our construction \eqref{bias ancestry} uses a different notion, leading to a coupling in which there is the original unbiased tree, plus immigrants, plus individuals descended from immigrants.}

The  Kesten-Stigum theorem concerns the following: Suppose we are given a Galton-Watson 
branching process with offspring variable $L$, whose 
distribution is given by $\p(L=k)=p_k$, with mean $m=\sum k p_k
\in (0, \infty)$, so that the number of individuals $Z_n$ at time $n$ has 
$\e Z_n = m^n$. The process given by 
$W_n := Z_n/m^n $ is a nonnegative martingale, hence converging almost surely to some limit $W$.
For $m \le 1$, it is easy to
prove that $Z_n \to 0$ a.s., so $W=0$; equivalently $\e W=0$. In particular the martingale is not
uniformly integrable.  But things are more subtle when $1 < m < \infty$.  For this case the
Kesten-Stigum theorem asserts that if $\e L \log L < \infty$, then $\e
W=1$, while if  $\e L \log L = \infty$, then $\e W=0$.  

The proof of the Kesten-Stigum theorem in 
\cite{LPP} begins with the observation that $W_n$ serves as
Radon-Nikodym derivative with respect to the usual distribution of $[T]_n$, the
branching process tree $T$ observed up to time $n$, of the distribution
size-biased by $Z_n$;  and since the resulting size biased distributions are
consistent, there results a notion of size-biased tree, which we call 
$T^*$.  Namely, this tree is obtained by picking  
one ``special'' individual from each
generation, and changing its offspring distribution from that of $L$ to that of the size-biased
version $L^*$, which satisfies, in particular, $\frac{1}{m} \e \log L^* =  \e L \log L$.  The Kesten-Stigum criterion is thus whether $\e \log L^*$ is finite or
infinite.  
Write
$Y_n$ for the number of {\em extra} children injected into generation $n+1$  by 
size-biasing the number of children of the selected special individual from generation $n$;   these individuals counted by $Y_n$ 
(but not their descendants), are called \emph{immigrants}.

It turns out that if $\e \log L^* < \infty$ then the process of immigrants grows
sub-exponentially, so the contribution from immigrants {\em and} their
descendants, up to time $n$, is $O(m^n)$. Then under the
size-biased distribution, $W<\infty$ a.s. and the size-biased law of $t$ is
absolutely continuous with respect to the original law.  On the other
hand, if 
 $\e \log L^* = \infty$ then the contribution  
 from immigrants, even
without counting their descendants, grows faster than any exponential;
in particular $W=\infty$ a.s. under the size-biased distribution, and
$W=0$ a.s. under the original distribution. Thus size-biasing plays a natural role in the understanding
of this result.  See \cite{LPP} for a proof and further information.


\subsection{Martingale size bias}\label{Martingale size bias}

Recall that in Section \ref{sect basics one}
we discussed size-biasing a process
 $\BX  
= (X_1,X_2,\ldots) \in [0,\infty)^\infty$ with joint law $\mu$,  by size-biasing 
one of its coordinates  $X_i$, assuming that
$a_i := \e X_i \in (0,\infty)$.
The recipe is given by \eqref{RN infinite}, and we wrote
${\bf X}^{(i)} = (X_1^{(i)},X_2^{(i)},\ldots)$ for the resulting process.
A natural question:  what is the result if $\BX$ is a martingale?   

So assume now that $\BX$ is a martingale,  nonnegative and nonconstant.  This implies, in particular, 
 that for each $i=1,2,\ldots$, the mean $a_i := \e X_i$ is in $(0,\infty)$, with $a_1=a_2=\dots$; call the common value $a$.   For any $i,n \ge 1$, the specifications  \eqref{RN infinite} of the distributions of ${\bf X}^{(i)}$ and ${\bf X}^{(n)}$, restricted to the first $n$ coordinates,  with Radon-Nikodym derivatives expressed in terms of arbitrary bounded measurable 
$g_n : [0,\infty)^n \to \BR$, are that
\begin{equation}\label{sizebias MG ni}
 \ \ \e  g_n(X_1^{(i)},X_2^{(i)},\ldots,X_n^{(i)})  =  \frac{1}{a} \ \e (X_i  \, g_n(X_1,X_2,\ldots,X_n)),
\end{equation}
and
\begin{equation}\label{sizebias MG nn}
 \ \ \e  g_n(X_1^{(n)},X_2^{(n)},\ldots,X_n^{(n)})  =  \frac{1}{a} \ \e (X_n  \, g_n(X_1,X_2,\ldots,X_n)).
\end{equation} 
By the martingale property of $\BX$, for all $i \ge n \ge 1$, the righthand sides of 
 \eqref{sizebias MG ni} and \eqref{sizebias MG nn} are equal to each other; hence
 \begin{equation}\label{sizebias MG consistent}
 \ \text{ for } i \ge n \ge 1, \ 
 (X_1^{(i)},X_2^{(i)},\ldots,X_n^{(i)})   =^d
 (X_1^{(n)},X_2^{(n)},\ldots,X_n^{(n)}) .
\end{equation} 
 
Now \eqref{sizebias MG consistent} says that we have a consistent family of finite dimensional distributions for a process, which we naturally call a {\em size-biased martingale}, and which we denote by
$\BX^*=(X_1^*,X_2^*,\ldots)$. The justification for this notation is that each individual coordinate 
$X_i^*$ of the process $\BX^*$  is a size biased version of $X_i$, in the sense of the 
original definition \eqref{sizebias}.  The proof, in turn, of this latter statement is that from 
\eqref{sizebias MG consistent} and the discussion in Section 
\ref{sect basics one} we must have $X_n^{(n)}  =^d X_n^*$,
while \eqref{sizebias MG consistent} applies for every $n$.
To recapitulate, the joint distribution of the first $n$ coordinates of 
$\BX^*$ is given by
\begin{equation}\label{sizebias MG n}
 \ \ \e  g_n(X_1^*,X_2^*,\ldots,X_n^*)  =  \frac{1}{a} \ \e (X_n  \, g_n(X_1,X_2,\ldots,X_n))
\end{equation} 
for all bounded measurable 
$g_n : [0,\infty)^n \to \BR$.   Considering the $n$-th coordinate marginally, 
the distribution of $X_n^*$ agrees with the elementary definition \eqref{sizebias} 
applied to $X_n$ in the role of $X$.
The martingale property of $\BX$ plays an essential role in this construction;  had $X_1,X_2,\ldots$  
been arbitrary non-negative random variables, each with strictly positive finite mean, 
while the size biased distributions for the $X_1^*, X_2^*,\ldots$ considered individually would still be 
given by \eqref{sizebias}, the \emph{joint} distribution 
is not specified  by \eqref{sizebias}.

Sometimes, one starts with a nonnegative process  
$\BZ=(Z_1,Z_2,\ldots)$ with means $a_i := \e Z_i \in (0,\infty)$,
in which the sequence $a_1,a_2,\ldots$, is not constant;
but after {\em scaling out the means} by defining $X_i := Z_i/a_i$,
the new process $\BX = (X_1,X_2,\dots)$ turns out to be a martingale.  
(An example of this is given by $Z_n := $ the size of the population at time $n$, in any Galton Watson process 
where the mean number of offspring per individual is $m \in (0,1) 
\cup (1,\infty)$.)    
Then $\BX$ can be size
biased as above, yielding $\BX^*=(X_1^*,X_2^*,\ldots)$. In light of
\eqref{scaling}, if we set $Z_i^* := a_i X_i^*$ for each $i$, the
distribution of this $Z_i^*$ necessarily agrees with the elementary definition
\eqref{sizebias} of size bias
applied to $Z_i$ in the role of $X$, but in addition the joint distribution of the 
\emph{size biased process} $\BX^*$ induces a joint distribution on $\BZ^* :=(Z_1^*,Z_2^*,\ldots)$,
i.e we have obtained a natural \emph{coupling} of the marginal size biased distributions.
To recapitulate:  given  nonnegative random variables
$Z_1,Z_2,\ldots$ with $\e Z_i \in (0,\infty)$,
and a joint distribution for a process $\BZ=(Z_1,Z_2,\ldots)$, 
if
$\BZ$ is a martingale, or the process derived from $\BZ$ by scaling
out the mean motion is a martingale, then there is a process $\BZ^*$,
simultaneously size biasing every coordinate.


Note that if we start with a martingale $\BX$, there is no particular reason for $\BX^*$ 
to be a martingale.  Similarly, the \emph{process with mean motion
  scaled out}, 
say $\BY$ with $Y_n := X^*_n/ \e X^*_n$, need not be a martingale.
However, there is an important class for which the martingale-based
process bias preserves structure: namely, if the process is {\em also} a Markov chain, 
then Markov structure is preserved.  We will prove this, in
Lemma \ref{lemma markov}.   

We limit ourselves to the case where the state space is
$\bbbz_+$, the nonnegative integers, for the sake of
easy notation, and also we limit ourselves to
the time homogeneous case;  neither of these restrictions is
essential.  To comply with the common convention for indexing time,  we switch the index set from
$\bbbn$, the natural numbers, to $\bbbz_+$.   And for later application to the 
special case of Galton Watson processes, we explicitly allow the
possibility that state 0 is a trap.   

\begin{lemma}\label{lemma markov}
Suppose $\BX = (X_0,X_1,\ldots)$ is a Markov chain on $S=\bbbz_+$ with
transition matrix $M$;  assume that $X_0=1$.   Suppose that 
$r_i := \sum_j  j \, M_{ij} < \infty$, for all $i \in S$.
Define a new stochastic matrix $N$, row by row, by size biasing the \emph{rows} of
$M$, if possible:
\begin{equation}\label{N from M}
  \text{ if } r_i>0, \text{ then } N_{ij} :=  j \, M_{ij}/r_i; \ \text{ if }
  r_i = 0, \text{ then } N_{ij} := M_{ij}.
\end{equation}
Let $\BY = (Y_0,Y_1,\ldots)$ be the Markov chain governed by $N$, with
$Y_0=1$.
Assume that for all $n, a_n := \e X_n \in (0,\infty)$.  Let $W_n :=
X_n / a_n$. 
If $(W_0,W_1,\ldots)$ is a martingale, then the  size biased process $\BX^*$
has the same distribution as the Markov chain $\BY$. 
\end{lemma}
Note: the process $\BY$ in  the statement of the Lemma 
may be considered as a special case of
\emph{Doob's $h$-transform}; see \cite[p. 296]{RW}.  
The letter  $h$ is mnemonic for \emph{harmonic}, and
 \eqref{N from M} is the special case where $h$ is the identity function --- as we pointed out, 
using the  fortuitous choice of the letter $h$ for the identity function, at the start of Section 
\ref{size particular}.
\begin{proof}
Fix a time $n$ and a sequence $z_0z_1\cdots z_n \in
S^{n+1}$, with $z_0 =1$.  We need to show that $\p(Y_0Y_1\cdots Y_n = z_0z_1\cdots z_n)$
$:= N_{1 z_1}N_{z_1 z_2}\cdots N_{z_{n-1} z_n}$
 is equal to 
$\p(X_0^*X_1^* \cdots X_n^* =z_0z_1\cdots z_n)$ 
$:= (z_n/\e X_n) \, \p(X_0X_1\cdots X_n = z_0z_1\cdots z_n)$
$=$ \newline $  (z_n/\e X_n) \, M_{1 z_1}M_{z_1 z_2}\cdots M_{z_{n-1} z_n}$. 
Observe that the martingale hypothesis implies that state 0 is a trap
for both processes, i.e., $M_{00}=N_{00}=1$, and that no other state
leads only to 0, i.e., for $i>0$, $M_{i0}<1$ and $N_{i0}<1$, hence
$r_i >0$ and $N_{i0}=0$.   So if some $z_k=0$, then $z_n=0$ and using $k$ as the
earliest index for which $z_k=0$, we have $N_{z_{k-1}z_k}=0$, hence 
$\p(Y_0Y_1\cdots Y_n = z_0z_1\cdots z_n)=0$
$= \p(X_0^*X_1^*\cdots X_n^* =z_0z_1\cdots z_n)$.  Otherwise, 
all $z_i \ne 0$, and the factor for time $k$, of the form    $N_{ij}$,
 is given by  $j \, M_{ij}/r_i$,
specifically with $i=z_{k-1}, j=z_k$.  To use the martingale property for
$W$, recall that $X_n = a_n W_n$, and note that $(X_{k-1}=i)$ is the same event 
as $(W_{k-1} = i / a_{k-1})$.  Hence  $r_i = \e (X_k | X_{k-1}=i) 
= \e ( a_k W_k | W_{k-1} = i / a_{k-1}) = i \, a_k/a_{k-1}$.
Hence the product 
$N_{1 z_1}N_{z_1 z_2}\cdots N_{z_{n-1} z_n}$ telescopes, to the
desired value.
\end{proof}

\subsection{Tree size bias}
 
 
Following \cite{LPP} 
`tree' 
will denote a rooted plane tree, possibly infinite, in which every individual has a finite number, possibly zero, of descendants.  We consider the set
$\cT$ of all  trees, and for $t \in \cT$  we 
let $[t]_n$ be the set of all trees whose first $n$ levels agree with $t$.
Write $\cT_n \subset \cT$ for the set of trees of height at most  $n$;
each $\cT_n$ is countable.
The sigma algebra $\cF_n$ on $\cT$ is generated by sets of the form
$[t]_n$, $t\in \cT_n$, 
and the sigma algebra $\cF$ is generated by the union of the $\cF_n$.
A probability distribution on $(\cT,\cF)$ can then be specified via a
consistent family of probability distributions on $\cF_n$, $n=1,2,\ldots$.
  
 Write $z_n(t) \ge 0$ for the number of individuals in level $n$ of $t$; for each fixed $n \ge 1$, this gives a nontrivial notion of size.  (By our convention that the tree is rooted,
 we always have $z_0(t)=1$.) 
  {\em Any} probability distribution on $\cT$, yielding  random trees $T$,
 can be size biased, giving a new distribution yielding trees $T^*$, 
\emph{provided} that with $Z_n := z_n(T)$,
 we have both $ \e Z_n \in (0,\infty)$    
 \emph{and} that the process ${\bf W} = (W_0,W_1,W_2,\ldots)$ with 
 $W_n := Z_n/ \e Z_n$
 is a martingale with respect to the filtration $\{\cF_n\}$. 
Specifically, for each $n$ we 
bias the distribution of trees 
of height at most $n$, via the following formula:
 for a given deterministic tree $t_n$ of height at most $n$, with
 $z_n \ge 0$ individuals at level $n$, we set
\begin{equation}\label{bias tree}
   \p(T^* \in [t_n]_n) ) = \frac{z_n}{\e Z_n} \  \p(T \in [t_n]_n) ) .
\end{equation}
The proof that the distributions are consistent, and that we have thus defined a tree-valued
process, depends on the martingale property of the process  ${\bf W}$.

\subsection{The size biased Galton Watson tree, with or without a spine}
Returning to Galton Watson trees, 
we would like to point out that passing from a Galton-Watson branching \emph{process}
to the associated random \emph{tree} depends not only  on the offspring 
distribution $(p_0,p_1,p_2,\ldots)$, but also an (often implicit) imposition of symmetry. 
Specifically, let ${\bf L} = (L_{n,i})_{n\ge 0, i \ge 1}$ be an
array with independent identically distributed entries, where $\p(L_{n,i}=k)=p_k$.  The usual recursive
construction for the process of \emph{counts}, in which $Z_0 :=1$, and then for $n
\ge 0$ we set
\begin{equation}\label{ancestry}
Z_{n+1} := \sum_{1 \le i \le Z_n} L_{n,i},
\end{equation}
 gives rise to a plane
tree if we declare that, for $i=1$ to $Z_n$, the $i$th individual in generation
$n$ has $L_{n,i}$ children. The distribution of the plane tree from
this standard construction has \emph{maximal}
symmetry:  at any time $n$, all $Z_n$ subtrees, rooted at an individual of generation
$n$, are equal in distribution to the original process.  

But alternatively, for certain purposes, given $Z_n=k$ we could sort
$L_{n,1},\ldots,L_{n,k}$ in nonincreasing order, now renamed
$A_{n,1} \ge \cdots  \ge A_{n,k}$, and then declare that the $i$th individual in generation
$n$ has $A_{n,i}$ children.   With this construction we would still have the same counts
$Z_{n+1}$ as before, and hence the same process $(Z_0,Z_1,Z_2,...)$,
but now a different tree lacking distributional symmetry. Namely, 
if a parent has more than one child, then his second-born
child is guaranteed to produce no more grandchildren than his first-born
child produces, and so forth.

By common agreement, the Galton Watson \emph{tree}  
is the one given by the first construction, with maximal symmetry,
rather than the one arising,
say, 
from sorted offspring counts.
To size bias 
this tree
let us once again start with \eqref{noniid}, which says that a sum of
independent (non-negative, finite nonzero mean) 
random variables is size biased by applying size bias to a
single summand.   In the spirit of maximal symmetry, we fix one particular
joint distribution for $(L,L^*)$, with $L^* \ge L$ always; see 
 \eqref{lies above}.  Then augment
the array ${\bf L}$ so that it becomes ${\bf L} =
((L_{n,i},L^*_{n,i}))_{n\ge 0, i \ge 1}$ whose entries are i.i.d. pairs,
but possibly with dependence within in each pair. (In Section \ref{sect inf div},   Theorem
\ref{thm inf div and Levy} says precisely when it is possible to have
$L$ and $L^*-L$ independent.)   

Continuing in the spirit of
\emph{maximum} symmetry, to construct the size biased tree $t^*$, 
one would naturally  start with i.i.d. uniform (0,1) random variables
$U_0,U_1,\ldots$, independent of ${\bf L}$, with $U_n$ used to decide \emph{which}
individual in generation $n$ will have a size biased number of
children.  A little more formally, the tree is constructed
recursively:
for each $n \ge 0$,
given the tree $t^*$ observed up to time $n$, with $Z_n^*$ individuals
at time $n$ (and, always $Z_n^* \ge 1$), take $I_n := \lceil Z_n^* U_n
\rceil$; then for $i \ne I_n$ the $i^{ \rm th}$ individual in generation
$n$ has $L_{n,i}$ children, but for $i=I_n$ the number
of children is $L^*_{n,i}$.   The resulting tree $t^*$ has 
\begin{equation}\label{bias ancestry}
Z_{n+1}^* := \sum_{1 \le i \le Z_n^*} L_{n,i} + ( L^*_{n,I_n} - L_{n,I_n}),
\end{equation}
and by \eqref{noniid}, this is equal in distribution to 
$(\sum_{1 \le  i \le Z_n^*} L_{n,i})^*$ --- note that the sum, being size-biased,
has $Z_n^*$ i.i.d. summands, rather than $Z_n$ summands as in \eqref{ancestry}.

\begin{exercise}\label{tree exercise 1}
Check that the distribution of the tree produced by the above
procedure has distribution satisfying \eqref{bias tree}.
\end{exercise}
\vspace{-1pc} 
\noindent To be complete, without spoiling the reader's fun, we 
supply a solution 
but postpone it until the end of Section \ref{sect proof dichotomy}.   
For historical reasons, we
hereby name the tree from 
above procedure as the \emph{spineless} (biased) tree.

In contrast to the maximal symmetry spineless procedure described
above, and following \cite{LPP},  for
$n\ge 1$, we could restrict the $n$th generation candidates for size
bias, instead of all $Z_n$ individuals,  to just the $S_{n-1} :=
L^*_{n-1,V_{n-1}}$ children of the individual $V_{n-1}$ in generation $n-1$
who was size biased.  So in this tree $V_0=1$ and, for $n >0$, $V_n$ is descended from
$V_{n-1}$; and the non backtracking path from the root,
$(V_0,V_1,V_2,\ldots)$ is called the \emph{spine} of the biased tree.
We refer to the tree in this construction as the \emph{spinal} (biased) tree.
  To recapitulate:  the spineless
tree has a list $(I_0=1,I_1,I_2,\ldots)$ of biased individuals, i.e., person $I_n$
in generation $n$ uses the distribution of $L^*$ to dictate his
unusually large number of children. By contrast, while 
the spinal tree has a similar list of biased individuals,
$(V_0=1,V_1,V_2,\ldots)$, with $V_n$ in generation $n$, 
{\em these} biased nodes form a
path.   The procedure with a spine can be traced back to \cite{Kesten} Kesten 1986, who was studying 
 critical GW processes, conditional on nonextinction, and used the term \emph{backbone} instead of \emph{spine}.

Here are two natural questions:

{\noindent} {\bf Question} 1  
Given a random spineless tree and a random spinal tree, without being
told which individuals are biased, can one tell which tree came from
which procedure?  

{\noindent} {\bf Question} 2 
Given a random spinal tree,
without being told where the spine is, can one identify the spine?
 
Question 2 is easily answered for subcritical or critical Galton Watson proceses, since in these cases there is a unique infinite path.   We will say more about these cases in Section \ref{sect tree survival}.

Initially, we found it hard to guess the answer to Question
1; it is \emph{not obvious} whether the spineless tree and the spinal tree
have the \emph{same} distribution. But since a computation  confirms that the 
spinal tree also satisfies \eqref{bias tree}, the answer to Question 1
is a definite \emph{no}:  while the two procedures have different joint
distributions for (tree, bias markers),  they have the same
marginal distribution for tree. 
\begin{exercise}\label{tree exercise 2}
Derive the marginal distribution for spinal trees.  Hint, 
it may help to show, at the same time, that conditional on the
tree up to time $n$, with $k = Z_n^*$ individuals at level $n$, 
the spinal position
 $V_n$  
 in generation $n$ 
 is
uniformly distributed from $1$ to $k$.
\end{exercise}
\vspace{-1pc}
\noindent  The marginal distribution for spinal trees 
is derived in \cite{LPP}, and we give a derivation, 
at the end of Section \ref{sect proof dichotomy}.
  
The answer to Question 2, for supercritical processes, depends, as does the Kesten-Stigum result, on
whether $\e L \log L$ is finite or infinite.   We spend the rest of this section on this dichotomy, and then return, in section \ref{sect tree survival}, to consideration of the subcritical and critical cases.

\begin{theorem}\label{thm spine}
Consider a supercritical Galton Watson process, with offspring distribution $L$ having $\e L < \infty$, and the size biased tree generated by the spinal procedure.   Given the tree alone:
\begin{enumerate}
\item If $\e L \log L = \infty$, the spine can be correctly identified, with probability 1.
\item If $\e L \log L < \infty$, any procedure to find the  spine fails, with probability 1.
\end{enumerate}
\end{theorem}
 
Before giving the proof, we remark that case (1) is easy, because  $\e L \log L = \infty$ implies that the distributions of tree and size-biased tree are mutually singular.   In the other case, with $\e L \log L < \infty$, the distribution of biased tree is absolutely continuous with respect to the unbiased tree, but some work is needed.
We need something similar to Fano's inequality, giving a lower bound on the error probability for classification, but Fano requires the Kullback-Liebler divergence to be finite.   So we are led to prove two lemmas about selection in a general setting.

\subsection{Selecting one special item out of $k$ choices, assuming $Q  \ll P$}

We want to detect the one item sampled from $Q$, when mixed in with $k-1$ others sampled from $P$.

Lemmas \ref{too obvious} and \ref{lost in noise} below are stated and proved for a fairly general pair of 
distributions $P$ and $Q$ satisfying $Q \ll P$, meaning that $Q$ is dominated by $P$, i.e., $Q$ is absolutely continuous with respect to $P$, i.e., $P(A)=0$ implies $Q(A)=0$.  We will apply Lemma \ref{lost in noise} to a Galton Watson tree for which $\e L \log L $ is finite and $m > 1$; 
then  $P$ will be the GW tree law, and $Q$ the size-biased law.   We remark that though case 1) in the proof of  Lemma \ref{lost in noise} cannot occur in the Galton Watson situation, nonetheless we prefer to have Lemma \ref{lost in noise} in its natural generality.

\noindent  {\bf Setup for selecting the special one,  out of $k$ choices.}\label{sect gold coin}
Fix $k>1$.   Let $P,Q$ be laws on  a Polish space $S$. 
 Let $Y_1,\ldots,Y_k$ be independent, with $Y_1$ sampled from $Q$, and $Y_2,\ldots,Y_k$ sampled from $P$, and let $X_1,\ldots,X_k$ be obtained from the $Y$s by an independent 
 uniformly distributed
 random permutation $\pi \in \mathcal{S}_k$.  So $X_i = Y_{\pi(i)}$, and then $I := \pi(1)$ identifies the index of the $X$ value sampled from $Q$; one might write $\BX = {\bf Y} \circ \pi$.  A \emph{selection procedure} is a function
$f:  S^k \to [k]$,  
meant as a guess of $I$ as a function of  
the sample $\BX = (X_1,\ldots,X_k)$.  The \emph{score} 
for a selection procedure $f$ is  $s(f) := \p(f(\BX)=I)$.

In the case $Q \ll P$, it is ``obvious" that the best selection procedures, i.e., those with maximal score, are precisely those which inspect the likelihood ratio $r(x) = dQ / dP (x)$, and pick $I$ arbitrarily from those indices $i$ which maximize $r(X_i)$, relative to the $k$-sample.   To prove this, while taking into account possible ties, we define a particular candidate $f_0$ for best selection procedure, by picking the \emph{earliest} index among those $i$ 
for which $ r(X_i)= \max(r(X_1),\ldots,r(X_k))$.

\begin{lemma}\label{too obvious} {\bf Optimal selection}.
In the setup above, \emph{any} selection procedure $f:  S^k \to [k]$ satisifes
$$
s(f) \le s(f_0),   
$$
and furthermore   $s(f) = s(f_0)$   
implies that with  $J := f(\BX)$   
we have 
$r(X_J) = \max(r(X_1),\ldots,r(X_k))$ with probability one.
\end{lemma}
\begin{proof}
Write ${\bf x} = (x_1,\ldots,x_k)$ and $z=r(x_1)+\cdots + r(x_k)$.    Conditional on $\BX = {\bf x}$, 
the odds [$\p(I=1):\p(I=2): \cdots : \p(I=k)$] are equal to
[$r(x_1):r(x_2): \cdots : r(x_k)$], hence $\p(I=i |  \BX = {\bf x} ) = r(x_i)/z$.  Thus $\p ( f_0(\BX)=I |  \BX = {\bf x})
=  \max( r(x_1),\ldots,r(x_k))/z$;  any competing procedure $f$ has 
$\p ( f(\BX)=I |  \BX = {\bf x}) =  r( x_{f(\bf x)})/z$, with 
the same denominator $z$, and no larger a numerator, hence \\ 
 \mbox{$\p ( f(\BX)=I |  \BX = {\bf x}) \le \p ( f_0(\BX)=I |  \BX = {\bf x})$}
 with equality holding if $r(x_{f(\bf x)})=r(x_{f_0(\bf x)})$.  Taking expectation yields
the inequality $s(f) \le s(f_0)$, and the claim regarding when $s(f)=s(f_0)$. 
\end{proof}  
   
\begin{lemma}\label{lost in noise} {\bf Lost in the noise}.  Let $P$ and $Q$ be distinct probability distributions on $S$, with $Q$ dominated by $P$.  
Given 
$\varepsilon >0$, there exists $k_0 < \infty$ such that for all $k \ge k_0$, in the setup above, with $Y_1$ distributed according to $Q$ and $Y_2,\ldots,Y_k$ distributed according to $P$, every
selection procedure $f$ has $s(f)<\varepsilon$.
\end{lemma}
\begin{proof}
Using Lemma \ref{too obvious}, we may assume that the selection procedure is $f_0$, choosing an item of maximal likelihood ratio $r$.

Take any version $r$ of the Radon-Nikodym derivative, $r(x) = dQ/dP(x)$; our hypothesis implies that 
$\e r(Y_2) = 1$, and
$r$ is $Q$-almost surely finite, i.e.,  $\p(r(Y_1)<\infty)=1$.   Let $u \in (0,\infty]$ be the essential sup of $r$;
we get the same essential sup with respect to $P$ and with respect to $Q$.
We deal with two separate cases. Informally, in case 1, the special item might achieve
the maximal value, i.e.,  $r(Y_1)=u$, but even so, there is likely to be a many-way tie against noise.
 In case 2, the value $u$ is unobtainable, and most likely, some
nonspecial choice strictly beats the special, i.e., there is some random
$J \in [2,k]$, with $r(Y_J) > r(Y_1)$. More formally:

Case 1: 
  $p := \p(r(Y_1)=u)>0$.  Note, this implies that $u<\infty$,
hence $\p(r(Y_2)=u) = p/u > 0$.  
So we pick $k_0$ so that, if $N$ is  distributed
Binomial($k_0-1,p/u)$,  then
$\p(N \le 2/\varepsilon) < \varepsilon/2$.   Hence the event that
no more  
than $2/\varepsilon$ items in the sample have $r(X_i)=u$
contributes at most $\varepsilon/2$ to $\p(f(\BX) \ne I)$, and on the
complementary event, by exchangeability, the conditional probability of
picking $I$ correctly is less than $\varepsilon/2$.

Case 2: 
 $ \p(r(Y_1)=u)=0$.  We can pick $t<u$ so that $q := \p(r(Y_1) \ge t) \in (0,\varepsilon/2)$.   Note that
$\p(r(Y_2) \ge t) \ge q/u >0$.   When $r(Y_1)<t$ and $k$ is large, with high
probability at least one of the $k-1$ items generated from $P$ will
have a higher value for $r$  than that of  the item generated from Q.  That is, by taking $k_0$ large enough, 
so that $(1-q/u)^{k_0-1} < \varepsilon/2$,
we can guarantee that for all $k\ge k_0$,
$\p( $ at least one of $Y_2,\ldots,Y_k$ has $r(Y_i) \ge t) > 1 - \varepsilon/2$.   Hence a procedure, like $f_0$, 
that picks one item from among those achieving $\max(r(X_1),\ldots,r(X_k))$, has probability at least $1-\varepsilon$ of picking    incorrectly.   
\end{proof}
   
\subsection{Proof of the spinal identification dichotomy} 
\label{sect proof dichotomy}  

\noindent  {\bf Proof of Theorem \ref{thm spine}}.
Identification of the  spine $(V_0,V_1,V_2,\ldots)$ is the conjunction of identifying $V_n$, for $n\ge 0$.  So with respect to the 0--1 dichotomy,  for (1) it suffices to show that for each $n$, $V_n$ can be correctly identified
with probability 1, while for (2), it suffices to show that for arbitrary  $\varepsilon > 0$, there exists
 $n = n(\varepsilon)$ for which the probability of correct identification of $V_n$  is less than $\varepsilon$.

For the spinal tree, observed up to time $n$, and conditional on the event $Z_n^*=k$, the distribution of the $k$ rooted subtrees
with roots at time $n$ fits exactly the setup described in Section \ref{sect gold coin}:  one of the trees is distributed
 according to $Q$, the law of the size biased GW tree, the other $k-1$ are distributed according to $P$, the unbiased GW tree law, all $k$ are mutually independent, and, using the uniformity of $V_n$ in $[k]$, as proved in \cite{LPP}, the joint distribution of these $k$ trees matches that of the random permutation $\pi$ applied to an ordered sample $Y_1,\ldots,Y_k$ in which $Y_1$ has the distribution $Q$.

For (1), with $\e L \log L = \infty$, Theorem A in \cite{LPP} includes the statement that $P$ and $Q$ are mutually singular.   So pick a subset $A \subset \cT$ having $P(A)=0,Q(A)=1$, and given the $k$ rooted subtrees, pick the node whose subtended tree lies in $A$, 
thereby finding $V_n$ correctly with probability 1.

For (2), with $\e L \log L < \infty$, Theorem A in \cite{LPP} includes the statement that $Q \ll P$.  Given
$\varepsilon > 0$, apply Lemma \ref{lost in noise} to find a value $k_0$ that works for $\varepsilon/2$.  Then
use the supercriticality to find a single  value $n$ for which $\p(Z_n \ge k_0 | Z_n > 0) > 1 - \varepsilon/2$.   The combination shows that,
given all the subtrees rooted at time $n$, the chance of correctly picking the one whose root is $V_n$ is less than $\varepsilon$.  Then, since conditional on the tree up to time $n$, the location of $V_n$ was uniform from 1 to $Z_n^*$, and using conditional independence of past and future, given the present, any function applied to the entire tree, attempting to identify $V_n$, has probability less than $\varepsilon$ of gettng the correct value.
\\
QED    
\\ \\ 

\noindent  {\bf Answers to the exercises on spineless and spinal trees}.  Suppose the given tree $t_n$ of height at most $n$ has $z_j$ nodes at height $j$, for $j=0$ to 
$n$.   Write $GW(t_n)$ for the Galton Watson probability that the tree observed up to time $n$ matches  this tree, corresponding to the last factor on the right side of \eqref{bias tree}.   For both the spineless tree and the spinal tree, we calculate the joint distribution of tree and bias markers up to time $n$, then sum over possible locations of the bias markers, to show that the marginal distribution of tree satisfies \eqref{bias tree}.

\noindent  {\bf Answer to Exercise \ref{tree exercise 1}}
For the spineless procedure, recall our notation $I_n$ designating which individual 
in generation $n$ gets biased,  and write ${\bf I} = (I_0, I_1,I_2, \ldots,I_{n-1})$
for the process naming those individuals arising in forming a tree up
to time $n$.  The possible values for ${\bf I}$ form a set $S$, with $|S| = z_0 z_1 \cdots z_{n-1}$, and using the notation $[k] := \{1,2,\ldots,k \}$, $S$ is the Cartesian product
$S = [z_0] \times [z_1] \times \cdots \times [z_{n-1}]$.   Given both $t_n$ and ${\bf i} \in S$, so that we 
specify  
the tree up to time $n$, \emph{and} which nodes were biased, at stages 0 to $n-1$, 
write $k_0,k_1,\ldots,k_{n-1}$ for the respective offspring counts for the biased nodes.   For the joint probability of tree and bias markers, taking into account first size bias factors of the form  $\p(L^*=k)/\p(L=k) = k/m$, and then factors of the form $1/z_j = \p(I_j=i_j)$,  and using $z_0=1$,   we have
$$
 \p(T^* \in [t_n]_n), {\bf I} = {\bf i} ) = \frac{k_0}{m} \, \frac{k_1}{m} \cdots \frac{k_{n-1}}{m}
  \ GW(t_n) 
  \ \frac{1}{z_1 z_2 \cdots z_{n-1}}.
$$
When we sum over $S$ to get the marginal distribution of tree up to time $n$, the sum factors as a product indexed by time $j=0$ to $n-1$,
and the $k_j$ values  sum to $z_{j+1}$, explicitly 
$k_{j,1}+k_{j,2}+\cdots +k_{j,z_j}=z_{j+1}$, yielding
$$
\p(T^* \in [t_n]_n) = \sum_{{\bf i} \in S} \p(T^* \in [t_n]_n), {\bf I} = {\bf i} ) = \frac{z_1}{m} \, \frac{z_2}{m} \cdots \frac{z_{n}}{m}
  \ GW(t_n) 
  \ \frac{1}{z_1 z_2 \cdots z_{n-1}}.
$$
$$
= \frac{z_n}{m^n} \ GW(t_n),
$$ 
which is \eqref{bias tree} in the Galton-Watson case.

\noindent  {\bf Answer to Exercise \ref{tree exercise 2}}
For the spinal procedure, write $V_0, V_1,\ldots, V_{n}$ for the  random 
spine, for the tree restricted up to time $n$.  Given the tree, the initial segments of spine,
 down through times $0,\ldots,n$, are   in one-to-one correspondence with the nodes 
$V_0, V_1,\ldots, V_{n}$
along the spine.  Given both $t_n$, and $v_n \in [z_n]$ to serve as the value of $V_n$, so that we 
specify  
the tree up to time $n$, \emph{and} which nodes were biased, yielding a path $v_0,v_1,\ldots,v_n$ from root to $v_n$,
write $k_0,k_1,\ldots,k_{n-1}$ for the respective offsping counts for the biased nodes.   For the joint probability of tree and bias markers, taking into account first size bias factors of the form  $\p(L^*=k)/\p(L=k) = k/m$, and next factors of the form $1/k_j = \p(V_{j+1}=v_{j+1})$,    we have
$$
 \p(T^* \in [t_n]_n), V_n = v_n)  ) = \frac{k_0}{m} \, \frac{k_1}{m} \cdots \frac{k_{n-1}}{m}
  \ GW(t_n) 
  \ \frac{1}{k_0 k_1 \cdots k_{n-1}}= \frac{1}{m^n} \ GW(t_n).
$$
Summing over the $z_n$ possible values for $v_n$, to give the marginal distribution of tree up to time $n$, shows that the spinal tree satisfies \eqref{bias tree}.
   
\subsection{Subcritical and critical GW, conditional on survival forever}
\label{sect tree survival}

For a Galton Watson process, consider the event of survival forever, that is, $A :=\{ \forall n,Z_n > 0 \}$.
If the process is subcritical --- $0 < m <1$, or critical --- $m=1$,
then $\p(A)=0$. 
 Conditioning on $A$, by definition, means taking the limit, as $n \to
 \infty$, 
of conditioning on $Z_n>0$.  Athreya-Ney \cite[pp. 58]{AN}
prove that, when $\p(A)=0$, conditioning on $A$ achieves the same
distribution as size biasing, 
although their
result imposes the extra hyothesis that $p_1 := \p(L=1)>0$.  
In this section, we give an elementary proof, without the extra hypothesis.

\begin{lemma}\label{odds lemma 1}
  Suppose that $q$  and  $q(\t)$ are probability measures on $\bbbn$, (indexed by $\t \in \bbbn$ or $\t \in \bbbr$)
  so that, in particular, $q_j \ge 0$ for all $j \in \bbbn$, and $1 = \sum_{j \ge 1} q_j$.
  Let $S := \{ j : q_j > 0\}$ be the support of $q$, and assume that $S$ is also the support of $q(\t)$, for every $\t$.   Suppose that as $\t \to \infty$, the $q(\t)$-odds converge to the $q$-odds, that is
\begin{equation}\label{odds}
\forall  j, k \in S, \ \frac{q_j(\t)}{q_k(\t)} \to \frac{q_j}{q_k}.
\end{equation} 
Suppose also that 
\begin{equation}\label{odds tight}
\text{the family }  \{ q(\t) \} \text{ is tight}.
\end{equation} 
Then $q(\t) \to^d q$, equivalently,
\begin{equation}\label{odds conclusion}
\forall  j \in \mathbb{N},  \ q_j(\t) \to q_j.
\end{equation}
\end{lemma}
\begin{proof}
 By tightness, every subsequence of the $q(\t)$ has a subsubsequence with a limit.   
 Pick such a subsubsequence, and call its limit $p$.   Along this subsubsequence,
$q_j(\t) \to p_j$    and  $q_k(\t) \to  p_k$;   if $p_k > 0$ then $p_j / p_k  = \lim  q_j(\t) / q_k(\t)   =  q_j / q_k$,  for every 
$ j \in S$.   This implies  $p =q$, and of course all convergent subsequential limits being the same $q$ implies that $q(\t) \to^d q$ as $\t \to \infty$.
\end{proof}

\begin{lemma}\label{odds lemma 2}
Take the same setup as Lemma \ref{odds lemma 1}, 
assuming \eqref{odds}, but in place of \eqref{odds tight}, supposing
instead that 
\begin{equation}\label{odds monotone}
\forall  j,k \in S \text{ with } j>k, \  \frac{q_j(\t)}{q_k(\t)} \nearrow \frac{q_j}{q_k},
\end{equation} 
where the upward arrow denotes convergence upward.   Then \eqref{odds tight}
holds, so the conclusion \eqref{odds conclusion} holds. 
\end{lemma}
\begin{proof}
Using \eqref{odds monotone}, for any $k \in  S$
$$ \frac{ \sum_{j>k} q_j(\t) }{  \sum_{j \le k}  q_j(\t) }    \le       \frac{ \sum_{j>k} q_j }{  \sum_{j \le k}  q_j }.
 $$
 Given  $\varepsilon>0$,  pick $k \in S$ so that the right side above is less than $\varepsilon$.  This implies that 
 for every $\t$,  the left side is also less than $\varepsilon$, which
implies the tightness hypothesis \eqref{odds tight}.
\end{proof}

 \begin{theorem}\label{thm 7.7}   Let $\BZ$ be a subcritical or critical Galton Watson process, so that the offspring distribution $L$ has $m := \e L \in (0,1]$.   Then the size biased process, $\BZ^*$, is equal in distribution to $\BZ$ conditional on 
 survival forever;  equivalently,  as $n \to \infty$, $(\BZ | Z_n > 0) \to^d \BZ^*$.
\end{theorem}
\begin{proof}
The core of the proof is  the asymptotic relation,  for fixed $k$,   $(1-(1-\delta)^k) \sim k \, \delta$ as $\delta \to 0$.

Fix a time $n>0$ and a  value $i>0$ with $\p(Z_{n-1}=i)>0$.    We use Lemma \ref{lemma markov} with the Galton Watson process
serving as the Markov process whose transition matrix is $M$;  
row $i$ of $M$ gives the distribution of $Z_n$ conditional on $Z_{n-1}=i$, and size biasing leads to row $i$ of $N$ 
giving the distribution of $Z_n^*$ conditional on $Z_{n-1}^*=i$.   For use in Lemma \ref{odds lemma 2}, we take  $q$ to be the distribution on $\bbbn$ given by row $i$ of $N$, as specified by \eqref{N from M},  and we take $q(\t)$ to be the distribution of $Z_n$ conditional on ($Z_{n-1}=i$ and $Z_{n+\t} >0) $.
Writing $\delta(\t) := \p(Z_\t >0)$, the probability of survival for an
additional $\t$ units of time, starting from a population of size 1, we
have, with \emph{proportional to} denoted by $\propto$,
\begin{eqnarray}
q_k(\t) &:= &   \p(Z_n=k | Z_{n-1}=i,Z_{n+\t} >0) \nonumber  \\
&=& M_{ik} \, \frac{1 -(1-\delta(\t))^k}{1 -(1-\delta(\t+1))^i} \label{fred handle} \\
 & \propto & M_{ik} \ (1 -(1-\delta(\t))^k).  \nonumber
\end{eqnarray}
Using the hypothesis that the GW process is subcritical or critical, $\delta(\t) \searrow 0$ as $\t \to \infty$.  
This easily implies the hypothesis \eqref{odds monotone} for Lemma \ref{odds lemma 2}.  So $(\BZ|A)$
is a Markov process, whose transition matrix is $N$, and this process, called $\BY$ in 
Lemma \ref{lemma markov},
is equal in distribution to $\BZ^*$, using both  the martingale and Markov properties of GW, but not 
the full structure of GW, to enable 
Lemma \ref{lemma markov}.
\end{proof}

Next consider a subcritcal or critical Galton Watson tree, conditional
on survival forever.  Thanks to \eqref{bias tree}, combined with
Theorem \ref{thm 7.7}, it is ``obvious'' that the conditioned tree is
the size biased tree.  
A proof can be found in \cite{AbrahamDelmas}, and we now present a more direct proof.

 \begin{theorem}\label{thm 7.8}   Consider the tree $T$ for a subcritical or critical Galton Watson process, 
 so that the offspring distribution $L$ has $m := \e L \in (0,1]$.   
 The size biased tree $T^*$, as specified by \eqref{bias tree}, is equal in distribution to $T$ conditional on 
 survival forever;  equivalently,  as $n \to \infty$, $(T | Z_n > 0) \to^d T^*$.  
\end{theorem}
\begin{proof}
Fix $i$, consider the associated $q$ and $q(\t)$ from the proof of Theorem \ref{thm 7.7}, 
and  recall the notation $\delta(\t)$ for $\p(Z_\t >0)$, the probability of survival for an
additional $\t$ units of time, starting from a population of size 1. 
By formula \eqref{odds conclusion} and Lemma \ref{odds lemma 2} 
we know that $q_k(\t) \to q_k$ as $\t \to \infty$, for  $k$ in the support of $q$.  
Therefore, using \eqref{fred handle},  
we see that
$$
M_{i,k} \ \frac{k  \delta(\t)}{i  \delta(\t+1)} \to N_{i,j} = M_{i,k} \  \frac{k}{i \, m}
$$
Thus we must have $\delta(\t)/\delta(\t+1) \to 1/m$ as $\t \to \infty$,
and hence 
\begin{equation}\label{good fred}
\text{ for fixed } n, \ \delta(\t)/\delta(\t+n)  \to  1/m^n  \text{ as }  \t \to \infty. 
\end{equation}

Now for fixed $n >0$ and $t\in \cT$,  with $k:= z_n(t)$, using 
\eqref{good fred},
\begin{eqnarray*}
   \p(T \in [t]_n | Z_{n+\t}>0) & = &   \p(T \in [t]_n ) \ \frac{1 -(1-\delta(\t))^k}{\delta(n+\t)} \\
   & \to  &   \p(T \in [t]_n ) \ \frac{ k}{m^n},
\end{eqnarray*} 
 in the limit as $\t \to \infty$.
Comparison with \eqref{bias tree} completes the proof.
\end{proof}

\section{Size bias, tightness, and uniform
  integrability}\label{sect UI}

Recall that a collection of random variables $\{Y_\alpha: \alpha \in I \}$,
where $I$ is an arbitrary index set,
 is \emph{tight} iff for all $\varepsilon>0$ there exists $L<\infty$ such that
$$
\p(Y_\alpha \not \in [-L,L])  < \varepsilon \quad \mbox{for all $\alpha \in I$.}
$$
This definition looks quite similar to the definition of uniform integrability, where we say
$\{X_\alpha: \alpha \in I \}$ is \emph{uniformly integrable}, or UI, iff for all $\delta>0$ there exists $L<\infty$ such that $$
\e (|X_\alpha| ; X_\alpha \notin [-L,L])  < \delta \quad \mbox{for all $\alpha \in I$.}
$$
Intuitively, tightness for a family is that uniformly over the family, the probability mass due to large values
is arbitrarily small.  Similarly, uniform integrability is the condition that, uniformly over the family, the contribution to the expectation due to large values is arbitrarily small.
Since \emph{size bias} relates contribution to the expectation to probability mass, 
it should be possible to 
use size bias to express
a relation between uniform integrability and tightness.

\ignore{   THIS WAS BACKGROUND, for an audience of beginners, and now looks inappropriate --- RICHARD
Tightness of the family of random variables $\{Y_\alpha: \alpha \in I \}$ implies that every sequence of variables $Y_n$,  $n=1,2, \ldots$ from the family has a subsequence that converges in distribution.  Uniform integrability of the family  $\{X_\alpha: \alpha \in I\}$ implies that for  any sequence of variables with a distributional limit, say $X_{\alpha_k} \Rightarrow X$,  has additonally $\e X_{\alpha_k} \rightarrow \e X$.
}

We 
show,
in Theorems  \ref{thm:tui1} and \ref{thm:tui2}, 
that for random variables, i.e., real valued random \emph{elements}, there is an intimate connection between tightness and uniform integrability,  and that this connection is made via size bias.  But we must note, the concept of tightness is much broader than the concept of uniform integrability, in that tightness applies to random elements of metric and topological spaces, whereas uniform integrability is inherently a real valued notion. In more general spaces, to define tightness, the closed intervals $[-L,L]$ are replaced by arbitrary
 \emph{compact sets}, and the discussion below relates 
 only to
 metric spaces with the property that balls $\{x:  d(x,y) \le L \}$ are compact.

To discuss the connection between size biasing and uniform integrability, it is useful to restate the basic definitions in terms of nonnegative random variables. It is clear from the definition of tightness above that a family of \emph{nonnegative} random variables $\{Y_\alpha: \alpha \in I \}$ is tight iff for all $\varepsilon>0$ there exists $L<\infty$ such that
\begin{equation}\label{def tight}
\p(Y_\alpha > L)  < \varepsilon  \quad \mbox{for all $\alpha \in I$,}
\end{equation}
and from the definition of UI, that
a family of \emph{nonnegative} random variables  $\{X_\alpha: \alpha \in I \}$ is uniformly integrable iff for all $\delta>0$ there exists $L<\infty$ such that
\begin{equation}\label{def UI}
\e(X_\alpha ; X_\alpha > L)  < \delta \quad \mbox{for all $\alpha \in I$.}
\end{equation}
For general random variables, the family  $\{G_\alpha: \alpha \in I \}$ is tight [respectively UI]
iff $\{|G_\alpha|: \alpha \in I \}$ is tight [respectively UI].
Hence we specialize in the remainder of this section to random variables that are non-negative. 

\ignore{Since \emph{size bias} relates contribution to the expectation to probability mass, 
it should be possible to state a relation between uniform integrability and tightness.
 However, care 
 }
 
 Care must be taken to distinguish between the \emph{additive} contribution to expectation, and the \emph{relative} contribution to expectation. The following example makes this distinction clear. Let
$$
\p(X_n= n)=1/n^2, \p(X_n = 0) = 1 - 1/n^2, \quad n=1,2,\ldots.
$$
Here, $\e X_n = 1/n$, the family $\{X_n\}$ is uniformly integrable, but $1=\p(X_n^*=n)$, so the family
$\{X_n^*\}$ is not tight;  
the additive contribution to the expectation from large values of $X_n$ is small, but the \emph{relative} contribution is large --- one hundred percent!
The following two theorems, which exclude this phenomenon, show that tightness and uniform integrability are very closely related.

\begin{theorem}
\label{thm:tui1}
Assume that for $\alpha \in I$, where $I$ is an arbitrary index set, the random variables $X_\alpha$ satisfy
$X_\alpha \ge 0$ and $0 < \e X_\alpha < \infty$,  and let $Y_\alpha =^d X_\alpha^*$.
Then  
$$
\{ X_\alpha: \alpha \in I \} \mbox{ is UI \  if } \  \{ Y_\alpha: \alpha \in I \} \mbox{ is tight}.
$$
Assume further that the values $\e X_\alpha$ are uniformly bounded away from 0, say $c > 0$ and $ \forall \alpha,
c \le \e X_\alpha$.  Then
$$
\{ X_\alpha: \alpha \in I \} \mbox{ is UI \ iff } \  \{ Y_\alpha: \alpha \in I \} \mbox{ is tight}.
$$
\end{theorem}

\begin{proof} Since  $Y_\alpha =^d X_\alpha^*$, by \eqref{sizebias}, for every $L$ we have  $\p(Y_\alpha > L ) = \e ( 1(Y_\alpha > L) ) =
\e (X_\alpha 1(X_\alpha >L)) / \e X_\alpha$,  so 
$$
\e(X_\alpha ;X_\alpha >L) = \e X_\alpha \ \p(Y_\alpha > L ).
$$

First, we show that tightness implies UI.  Assume that $\{ Y_\alpha: \alpha \in I \}$ is tight, and take  $L_0>0$ to satisfy \eqref{def tight} with $\varepsilon = 1/2 $, so that $\p(Y_\alpha > L_0 ) <1/2$ for all $\alpha \in I$.  Hence, for all $\alpha \in I$,
$$
\e(X_\alpha ;X_\alpha >L_0) = \e X_\alpha  \p(Y_\alpha > L_0 ) <  \e X_\alpha / 2,
$$
and therefore,
\beas
L_0 \ge  \e(X_\alpha ;X_\alpha  \le L_0) &=& \e X_\alpha - \e(X_\alpha ;X_\alpha >L_0)\\
&>& \e X_\alpha - \e X_\alpha / 2 = \e X_\alpha / 2,
\enas
and hence $\e X_\alpha < 2 L_0$.
Now given $\delta >0$ let $L$ satisfy \eqref{def tight} for $\varepsilon= \delta/(2 L_0)$.
Hence $\forall \alpha \in I$,
$$
  \e(X_\alpha ;X_\alpha >L) = \e X_\alpha \ \p(Y_\alpha > L ) < 2 L_0 \  \p(Y_\alpha > L ) < 2 L_0 \ \varepsilon = \delta,
$$
establishing \eqref{def UI}.

Second we show that UI implies tightness, in the presence of means bounded uniformly away from zero . Assume that $\{ X_\alpha: \alpha \in I \}$ is UI, and let $\varepsilon >0 $ be given to test tightness in \eqref{def tight}. Let $L$ be such that \eqref{def UI} is satisfied with $\delta = \varepsilon c$.  Now, using
$\e X_\alpha \ge c$, for every $\alpha \in I$,
$$
 \p(Y_\alpha > L ) = \e(X_\alpha ;X_\alpha >L) / \e X_\alpha \le \e(X_\alpha ;X_\alpha >L) /c < \delta /c =\varepsilon,
$$
establishing \eqref{def tight}.
\end{proof}
As an alternate to Theorem \ref{thm:tui1}, for the sake of having cleaner hypotheses and a cleaner conclusion, we also give the following theorem.  Note below that the $X_\alpha$ to be involved in size bias are allowed to have $\e X_\alpha = 0$  --- it is not a typo --- because we will be taking $(X_\alpha +c)^*$ for some $c>0$.

\begin{theorem}
\label{thm:tui2}
Assume that for $\alpha \in I$, where $I$ is an arbitrary index set, the random variables $X_\alpha$ satisfy
$X_\alpha \ge 0$ and $\e X_\alpha < \infty$.  Pick any $c \in (0,\infty)$, and for each $\alpha$ let $Y_\alpha = (c+X_\alpha)^*$.
Then
$$
\{ X_\alpha: \alpha \in I \} \mbox{ is UI \ iff } \  \{ Y_\alpha: \alpha \in I \} \mbox{ is tight}.
$$
\end{theorem}
\begin{proof}  By Theorem \ref{thm:tui1}, the family $\{c+X_\alpha\}$ is UI iff the family $\{(c+X_\alpha)^*\}$ is tight.
As it is easy to verify that the family $\{X_\alpha\}$ is tight [respectively UI] iff the family $\{c+X_\alpha\}$ is tight [respectively UI], Theorem \ref{thm:tui2} follows directly from Theorem \ref{thm:tui1}.
\end{proof}

\section{Size bias, the lognormal, and Chihara--Leipnik}\label{sect lognormal}

In this section we review a construction due to Chihara in 1970,
\cite{chihara}, and  Leipnik in 1979, \cite{leipnik,leipnik91}, 
of a family of discrete distributions having the same moment sequence as the lognormal.  Durrett \cite{durrett}
presents this result
with the comment
``Somewhat remarkably, there is a family of discrete random variables
with these moments.''    We hope here to show that, from the point of view of size bias, this construction is  
\emph{natural} and \emph{inevitable}, but we can only speculate that
for the original discoverers, 
size bias played a role in the creative process,
perhaps via \eqref{phi'};  see \cite[page 332,  formula (16)]{leipnik91}.
As a reward for using size bias, we are able to show, in Theorem
\ref{thm ra}, that the
lognormal itself is a mixture of these discrete distributions, and
furthermore that these discrete distributions are the extreme points
of a Choquet simplex --- in this case, the set $U_c$ of solutions of
\eqref{times c}, which is a subset of the closed convex  set $V_c$
formed by all distributions
having the same moments as the lognormal $X=\exp(\sqrt{\log c}\, Z)$.   The results in this section,
linking
the lognormal distribution with size bias, appear in \cite{christiansen,pakes,pakes96}; see also \cite{lopez}.

Throughout this section, we write $Z$ for a  standard normal, with moment generating function
$M(\beta)=e^{\beta^2/2}$.  The \emph{standard lognormal} is given by
$X=e^Z$, with moments
\begin{equation}\label{lognormal moment plain}  
 \e X^n =  \e \exp(n \,  Z )=M(n) =
e^{n^2/2},
\end{equation}
for $n=0,1,2,\ldots$.  (It is  clear that \eqref{lognormal moment
  plain} holds for all $n \in (-\infty,\infty)$, but historically,
moments usually refer to the case $n=0,1,2,\ldots$.)
Similarly, for $\sigma>0$, the lognormal $X=e^{\sigma Z}$ obtained by exponentiating the normal with mean zero and variance $\sigma^2$ has moments $\e X^n =  \e \exp(n \, \sigma Z )=M(\sigma n) =
e^{n^2\sigma^2/2}$. 
Hence it is natural to define, taking $c=e^{\sigma^2} \in (1,\infty)$,
\begin{equation}\label{V}
   V_c := \{ \mu:   \mu=\cL(X) \text{ for some 
     } X \ge 0,\text{ with } \e X^n = c^{n^2/2}, n \ge 0
    \}.
\end{equation}

  The famous fact that $V_c$ is not a singleton set, i.e., that the lognormal distribution is not determined by
its moments, is from  Stieltjes
in  1894 
\cite[Section 56, page J. 106]{stieltjes1894}, reprinted in
\cite{stieltjes}. The family of examples in  \eqref{heyde both} is
also from Stieltjes \cite{stieltjes1894}, 
although probabilists,   e.g., \cite{durrett,feller2}, attribute 
it 
to Heyde, who rediscovered it in 1963 \cite{heyde}.
These alternate probability distributions having the same
moments as the lognormal are continuous, with density presented via a 
perturbation of the lognormal density, as follows.   We will write $f_{0,\sigma^2}$ for the density of the lognormal $X=e^{\sigma Z}$:
\begin{equation}\label{lognormal density}
     f_{0,\sigma^2}(x) = \frac{1}{x \sigma \sqrt{2 \pi}} \exp(-(\log x)^2/(2\sigma^2)),   \ x \in
     (0,\infty).
\end{equation}
For positive integers $m$ and real $\delta \in [-1,1]$ define
\begin{equation}\label{heyde}
     g_{m,\delta}(x) = 
       1 + \delta \sin( 2 \pi m \, \log x /\sigma^2) , 
      \ x \in    (0,\infty),
\end{equation}
so in case $\delta=0$, one has $g_{m,\delta}(x)=1$ for all $x$.
Let $h_{m,\delta}$ be given by
\begin{equation}\label{heyde both}
     h_{m,\delta}(x) = f_{0,\sigma^2}(x) \times g_{m,\delta}(x),
       \ x \in
     (0,\infty).
\end{equation}
One then checks that for integers $n$, $\int x^n h_{m,\delta}(x) \ dx = \int x^n f_{0,\sigma^2}(x) \ dx =
e^{n^2\sigma^2/2}$.
 In particular, the case $n=0$ shows that $h_{n,\delta}$, clearly a non-negative function, is a density.


Let $X = e^{Z}$, and consider its size biased
version, $X^*$.  By \eqref{moment shift} and \eqref{lognormal moment plain}, for integers $n$,
\begin{equation}\label{lognormal moment}
     \e (X^*)^n = \frac{\e X^{n+1}}{\e X} = \frac{M(n+1)}{M(1)} = 
     e^n M(n) = e^n \e X^n = \e (eX)^n.
\end{equation}
Of course, since the lognormal distribution is \emph{not}
characterized by its moments, this only \emph{suggests}, and does
\emph{not prove}, that $X^* =^d e X$.
Similarly, 
the general lognormal and its moments are given by
\begin{equation}\label{lognormal moments general}
X=\exp(\sigma Z + \mu), \ \e X^n = 
 e^{\mu n + \sigma^2n^2/2}
\end{equation}
and calculation of the moments of
$X^*$ \emph{suggests} that for   $X=\exp(\sigma Z + \mu)$ we have $X^* =^d
e^{\sigma^2} X$.  Simple computation with the
density  and \eqref{sizebias-f} shows that indeed, 
\begin{equation}\label{lognormal property}
X=\exp(\sigma Z + \mu) \text{ has } X^* =^d c X, \text{ with }
c=\exp(\sigma^2).
\end{equation}  
We leave the proof of \eqref{lognormal property} as an exercise for the reader, with our solution given by 
this\footnote{The density of $\exp(\mu +\sigma Z )$ is 
$f_{\mu,\sigma^2}(x) = 1/( x \sqrt{2 \pi} \sigma)
\exp(-(\log x - \mu)^2/(2 \sigma^2))$.
Expanding the square in the exponent, and keeping track only of factors that  vary with $x$, we have 
$f_{\mu,\sigma^2}(x) \propto (1/x) x^{-(\log x)/(2 \sigma^2)} x^{\mu/ \sigma^2}$. Hence for any real $\beta$, 
$ f_{\mu+\beta \sigma^2,\sigma^2}(x) \propto x^\beta f_{\mu,\sigma^2}(x)$.  The case $\beta=1$ shows that 
  $x f_{\mu,\sigma^2}(x)$ is proportional to $f_{\mu+\sigma^2,\sigma^2}(x)$, hence by \eqref{sizebias-f}, $(\exp(\mu +\sigma Z))^* =^d \exp((\mu+\sigma^2)+ \sigma Z)$.
}
 footnote.

As regards the distributional family, 
varying $\mu$ corresponds to 
scaling $X$, 
and $X \mapsto yX$ 
is a trivial transformation, 
so it makes sense to
study only the case $\mu=0$. But varying  $\sigma$
is nontrivial; it corresponds to taking ordinary powers, $X \mapsto
(X)^\sigma$.  
So, we fix $\mu=0$ and let  $\sigma>0$ be arbitrary.  
We
will write $c=\exp(\sigma^2)>1$;  alongside our standard notation, $a
= \e X$, for $X=e^{\sigma Z}$ we have $a = \e X = \sqrt{c}$.

For the remainder of this section, for  $c \in (1,\infty)$, we
investigate random variables satisfying
\begin{equation}\label{times c}
   X \ge 0,   \ \ \ \  \e X = \sqrt{c}, \ \ \ \ \ X^* =^d c X,
\end{equation} 
along with the corresponding set of probability distributions,
\begin{equation}\label{U}
   U_c := \{ \mu:   \mu=\cL(X), \text{ for some 
     random variable X, satisfying \eqref{times c}} \} .
\end{equation}
With this notation,  \eqref{lognormal moments general} and \eqref{lognormal property} assert that  $X=\exp(\sqrt{\log c}\, Z)$ 
satisfies \eqref{times c}, and its lognormal distribution is an
element of $U_c$.

As the first step in our investigation of \eqref{times c},
inspired by Feynman's maxim,\footnote{Feynman Lectures on Physics, Vol.~2,
   Chapter 12.1 and oft again, ``the same equations have the same
   solutions.''} we note that our considerations lead twice to a  homogenous system of equations, of the form
\begin{equation}
\label{simple 1}
\forall   n \in \BZ, s_{n+1}= a c^n s_n,   \quad \mbox{which has solution} \quad s_n=s_0 a^n c^{n(n-1)/2}.
\end{equation}

For the first instance of \eqref{simple 1},  write 
 $m_n := \e X^n$, with $m_1= \e X =a$, so the moment shift relation \eqref{moment shift} can be written
as $\e (X^*)^n = m_{n+1}/a$. Using \eqref{times c}, we have  $\e (X^*)^n= \e (cX)^n = c^n
 \, m_n$, hence 
\begin{equation}\label{recursion moment}
   m_{n+1} = ac^n \, m_n.
\end{equation}
Combining $m_0=1$ with the solution to \eqref{simple 1}, we have
 \begin{equation}\label{solution moment}
    m_n =  a^n c^{n(n-1)/2} =  c^{n^2/2} \text{ (using } a=\sqrt{c}),
\end{equation}
 for all $n \in \BZ$.  In summary, so far we have shown that $U_c \subset
 V_c$, i.e.,  any solution of \eqref{times c} has the same moments as the lognormal $e^{\sigma Z}$.

For the second instance of \eqref{simple 1}, if $X$ satisfying \eqref{times c} has any pointmass at some
$b>0$, then it must have pointmass at every point $b \, c^n$ for $n
\in \BZ$.  
With  the benefit of hindsight\footnote{Defining for example $s_n = \p(X=bc^{n})$ or  $s_n = 1/\p(X=bc^{n})$ or $s_n = \p(X=bc^{-n})$ does not lead directly to \eqref{simple 1}  --- the reader \emph{might} enjoy  trying these.}  we go doubly negative, and for $n \in \BZ$ 
define $p_n$ and $r_n$ by $r_n = 1/p_n = 1/\p(X=bc^{-n})$.  We 
have $p_{n+1}= \p(X=bc^{-n-1}) = \p(cX=bc^{-n}) = \p(X^*=bc^{-n}) =
(bc^{-n}/a)\p(X=bc^{-n}) =
 (bc^{-n}/a) p_{n}$, so that  
 \begin{equation}\label{recursion mass}
   r_{n+1} = (a/b)c^n \, r_n.
\end{equation}
This is  \eqref{simple 1} 
 with $r_n$ in the role of $s_n$ and $a/b$ in the role of $a$, so
quoting the solution, 
and using $a=\sqrt{c}$, we get
$
    p_0/p_n = r_n/r_0 = (a/b)^n c^{n(n-1)/2} = b^{-n} c^{n^2/2}.
$
Finally, replacing $n$ by $-n$ 
 in  $p_n = \p(X= bc^{-n})$,
we have, for $n \in \BZ$,
\begin{equation}\label{solution mass 0}
    \p(X=bc^n) =  b^{-n} c^{-n^2/2} \    \p(X=b).   
\end{equation}
   
With some fixed $c>1$ in mind, for any $b \in (0,\infty)$ we call 
the set $$\{ \ldots,b/c^2,b/c,b,bc,bc^2,bc^3,\ldots \}$$ the   
``orbit of $b$,''  for short, or to say it fully, the orbit of $b$ modulo multiplication by powers of $c$.
The language here  comes from the theory of a group acting on a set;  orbits are equivalence classes, and $(0,\infty)$ is a disjoint union of orbits.  For a set containing exactly one representative for each orbit, the natural choice is $[1,c)$.

If we want $X$ supported on a single orbit, that is, with $1=\sum_{n \in
  \BZ} p_n$, then we need
\begin{equation}\label{solution mass}  
  \p(X=bc^n) =  b^{-n} c^{-n^2/2}/t(b,c)  , \text{ where } t(b,c) := \sum_{m\in \BZ} b^{-m} c^{-m^2/2}.
\end{equation}
The function $t$ is essentially the Jacobi {theta} function; the convergence of the series, for any $c>1$, is obvious.

However, the calculation connecting \eqref{solution mass 0} with
\eqref{times c} was done \emph{assuming} that $\e X  = \sqrt{c}$, and we will only have succeeded, in getting a random variable
with $X^* =^d cX$ and supported on a single orbit, if, and only if, it turns out that, under the mass function \eqref{solution mass},
one has $\e X = \sqrt{c}$. (It is trivial to check that if
\eqref{solution mass 0} \emph{and} $\p(X \in \{
\ldots,b/c^2,b/c,b,bc,bc^2,bc^3,\ldots \})=1$ \emph{and} $\e X =
\sqrt{c}$, then \eqref{times c} is true.)  
 So crossing our fingers we calculate, from \eqref{solution mass},
$$
\e X / \sqrt{c} = \sum_{n \in \BZ} b c^n \  b^{-n} c^{-n^2/2} \
  c^{-1/2} / \, t(b,c) = 1,
$$
with the change of variables $m=n-1$ justifying the final equality.

 
The above discussion shows how the use of size bias, particularly
\eqref{times c}, 
makes it relatively straightforward to rediscover and prove the
following theorem of Chihara and  Leipnik:
 
 \begin{theorem}[Chihara -- Leipnik]\label{thm roy}
 For any $\sigma >0$, with $c := \exp(\sigma^2)$, and for any $b \in (0,\infty)$, there is a distribution $\ell(b,c)$ for a discrete random variable $X_{b,c}$, whose support is the single orbit  $\{ \ldots,b/c^2,b/c,b,bc,bc^2,bc^3,\ldots \}$, with probability mass function given by \eqref{solution mass}.
 This random variable satisfies \eqref{times c}, which implies that 
 for $n \in \BZ$, $\e X_{b,c}^n  = \exp(n^2 \sigma^2/2)$, so taking $n \ge 0$ in particular,  the discrete random variable $X_{b,c}$ 
 has the same moments as the lognormal  $ \exp(\sigma Z)$, where $Z$ is standard normal.
 \end{theorem}

Another issue is whether the lognormal can be expressed as a mixture
of these discrete distributions.   
\ignore{
with italics to show Leipnik's exact words, and ordinary text to show our paraphrase\footnote{where we also switch to our notation, $b$ to repesent the base of the orbit, and $\ell(b,c)$ for the distribution on that orbit}, 
}
Leipnik 1991, \cite[page 337]{leipnik91},
wrote\footnote{Italics to show Leipnik's exact words, and ordinary text to show our paraphrase.} ``\emph{One hopes that for some mixing distribution} $d h(b) $ \emph{we have} that the lognormal distribution for $e^{\sigma Z}$
is a mixture, governed by $h$, of the single orbit distributions, \emph{and so too}
$$
   \phi(t) = \int_0^\infty \phi_b (t)  \ dh(b).
$$   
[The display above expresses the characteristic function of the lognormal as a mixture of the characteristic functions of the distributions $\ell(b,c)$.]
\emph{Unfortunately, the necessary $d \, h(b)$ 
is somewhat complicated and hence sheds little light on the sum distribution problem.  However, the extraordinary non-uniqueness of  the lognormal moment problem is apparent.} 

\ignore{We interpret Leipnik's words ``One hopes'' to mean that he \emph{conjectured} 
that the lognormal distribution is a mixture of single orbit discrete distributions.  But the sentence ``Unfortunately $\cdots$'' suggests that maybe Leipnik also proved the conjecture, but considered the result too messy to write up.   We have proved his conjecture, as given by  Theorem \ref{thm ra} below.  From our point of view, the obstruction Leipnik encountered might have been taking mixtures indexed by $(0,\infty)$, failing  to notice that $\ell(b,c) = \ell(bc,c)$. 
}
The words ``one hopes'' signal a conjecture; the sentence beginning   ``Unfortunately $\cdots$''
suggests that he may have had a proof too messy to publish.  Whatever the case, we supply a
proof here, in the form of Theorem \ref{thm ra} below.  Conceivably, the complication encountered by Leipnik might have arisen from considering mixtures indexed by $(0,\infty)$, without
exploiting the formula $\ell(b,c) = \ell(bc,c)$ --- the proof of which we leave as an
exercise for the reader.
It is natural, and simple, to take mixtures indexed by $[1,c)$; then there is a unique choice for $h$, with one simple computation to check.
For notation, we follow Leipnik, and write $dh(b)$ to denote  a general measure $h$ to govern a mixture; so that  $h$ may be discrete, absolutely continuous, singular continuous, or a mixture of these.  In the special case in
Theorem \ref{thm ra} given by \eqref{def h}, expressing the lognormal as a mixture of the $\ell(b,c)$, we
have $h$  absolutely continuous, with
density $h_c$ with respect to Lebesgue measure.

We show how to express the lognormal as a mixture of the Chihara--Leipnik discrete distributions $\ell(b,c)$ from Theorem \ref{thm roy}, via   Lemma \ref{lemma Leipnik}, Lemma
\ref{lemma times c}, and Theorem \ref{thm ra}.  
There is a related result, expressing a particular
continuously distributed random variable, not the lognormal, but having
the same moments, as a mixture of
these discrete distibutions, in \cite[Proposition
2.2]{berg-discrete}.

\begin{lemma}\label{lemma Leipnik}
Fix $c>1$.  
For any probability measure $h$ on $[1,c)$,
the mixture of the laws $\ell(b,c)$, governed by $dh(b)$, gives a distribution for $X$ which satisfies \eqref{times c}.
The set $U_c$ of distributions which satisfies \eqref{times c} is closed and convex, hence  \emph{any} mixture of distributions which satisfy \eqref{times c} is also a distribution which satisfies
\eqref{times c}.
\end{lemma}
\begin{proof}
Since for each $b \in [1,c)$, $m(b) := \e X_{b,c} = \sqrt{c}$, we are in the situation for Lemma 
\ref{lemma mixture} where the measure $h'$ governing $X^*$ as a mixture of the $X_{b,c}^*$ is the same as the original $h$,
governing $X$ as a mixture of the $X_{b,c}$.  Hence \eqref{times c} holds, since, obviously, scaling respects mixtures, i.e.,  the law of $c X$ is the mixture, governed by $h$, of the laws of $ c \, X_{b,c}$.  

That $U_c$ 
is \emph{closed} is a bit subtle.  Assume we are given $X_1,X_2,\ldots$ with each $X_n$ satisfying
\eqref{times c}, and that  $X_n \Rightarrow Y$.  Obviously $c X_n \Rightarrow c Y$, and 
Theorem 
\ref{limit theorem}  asserts that $X_n^* \Rightarrow Y^*$,  which combined with \eqref{scaling} and \eqref{times c} gives $Y^* =^d c Y$.  But \eqref{times c} also demands that $\e Y = \sqrt{c}$, so one must know that the family $\{X_1,X_2,\ldots \}$ is \emph{uniformly integrable}.  Fortunately, 
\eqref{solution moment} implies that $\e X^2 = c^2$ for  any solution of \eqref{times c}, which implies that the family is uniformly integrable.  

Finally, for the convexity of $U_c$, just as with mixtures of the $\ell(b,c)$,  Lemma \ref{lemma mixture}
applies, with the same measure governing $X$ as  mixture of solutions $X_\alpha$, governing $X^*$ as a mixture of the $X_\alpha^*$, and $cX$ as a mixture of the $c X_\alpha$. 
\end{proof}

\ignore{
For the following lemma,
which will imply that any two solutions of \eqref{times c}, which agree on an interval of the form $[x,xc)$,
must agree everywhere,  we give the hypotheses in slightly ore general form --- since we don't understand the consequences of \eqref{times c} with the stipulation $\e X = \sqrt{c}$ removed.
}

\begin{lemma}\label{lemma times c}
Suppose $c>1$, and $X,Y$ are positive random variables which satisfy
$$ 
  0 < \e X = \e Y < \infty  \text{ and }X^* =^d cX, Y^*=^d cY  .
$$    
If the laws of $X$ and $Y$, both restricted to $[1,c)$ agree, even only up to a constant mass factor $k \ge 0$, i.e., if
\begin{equation}\label{hyp k}
\text{ for all measurable } A \subset [1,c), \  \p(X \in A) = k \, \p(Y \in A).
\end{equation}
then $X =^d Y$.  (The case $k=0$ is specifically included in the hypothesis \eqref{hyp k}, but in every case, the conclusion implies that $k=1$.)
\end{lemma}
\begin{proof}
Let $a := \e X$, so by hypothesis, we also have $a= \e Y$, (but unlike
\eqref{times c}, 
we are \emph{not} assuming that $a=\sqrt{c}$).
Let $S(n)$ be the statement that for all bounded measurable $g$ which vanish 
outside $[c^n,c^{n+1})$, we have $\e g(X) = k \, \e g(Y)$. The hypothesis \eqref{hyp k} clearly implies the statement $S(0)$. 
Assume now that $S(n)$ holds.  Given a bounded measurable function  $g$ which vanishes off of $[c^{n+1},c^{n+2})$, we \emph{define} new functions  $g',g''$ by $g'(x) = g(x)/x$ and $g''(x)=g'(cx)$.  Clearly $g''$  is bounded, and vanishes off of $ [c^n,c^{n+1})$.
We have
$$
  \e g(X) = \e (X g'(X)) = a \, \e g'(X^*) = a \,   \e g'(cX) = a \, \e g''(X)  
$$
and similarly $\e g(Y) = a \, \e g''(Y)$.    Invoking $S(n)$ for the function $g''$, we get 
\begin{equation}\label{shifted}
   \e g(X) =  a \, \e g''(X)  = ak \, \e g''(Y) = k \e g(Y),
\end{equation}
hence $S(n)$ implies $S(n+1)$.

A similar argument shows that $S(n)$ implies $S(n-1)$.  In detail, given a bounded measurable function  $g$ which vanishes off of $[c^{n-1},c^{n})$, we \emph{define} new functions  $g',g''$ by $g'(x) = g(x/c)$, so that $g'(cx)=g(x)$,  and $g''(x)=xg'(x)$.  Clearly $g''$  is bounded, and vanishes off of $ [c^n,c^{n+1})$.
We have
$$
  \e g(X) = \e ( g'(cX)) =  \e g'(X^*) = \frac{1}{a}   \e (Xg'(X) )= \frac{1}{a} \e g''(X)  
$$
and similarly  $\e g(Y) = (1/a) \e g''(Y)$; hence \eqref{shifted} holds exactly as before, but this time showing that $S(n)$ implies $S(n-1)$.

Finally, knowing $S(n)$ for all $n \in \BZ$ implies that for bounded measurable $g$, $\e g(X) = k \, \e g(Y)$, and the special case $g=1$ shows that $k=1$, and hence $X =^d Y$.
\end{proof}

The following Theorem \ref{thm ra} applies in particular to the case where $X$ has the lognormal distribution with density $f(x) =$ $ 1/( x \sqrt{2 \pi} \sigma)
\exp(-(\log x)^2/(2 \sigma^2))$, recalling that with $c=\exp(\sigma^2)$,  $X$ satisfies \eqref{times c}.

\begin{theorem}\label{thm ra}
Let $X$ be any
positive random variable which satisfies \eqref{times c}.  Then there is a unique probability measure $h$ on $[1,c)$
such that the distribution of $X$ is the mixture, governed by $dh(b)$, of the 
Chihara--Leipnik single orbit distributions $\ell(b,c)$ of Theorem \ref{thm roy}, 
with point mass functions (and Jacobi theta function $t$) given by
\eqref{solution mass}.  The measure $h$ governing the mixture is specified as follows:  let $B$ be distributed as  $X$, conditional on $(X \in [1,c))$. Then the probability measure $h$ has Radon-Nikodym derivative, relative to the distribution of $B$, given by
\begin{equation}\label{mixture RN}
 \frac{h(db)}{\p(B \in db)} = \frac{t(b,c)}{\e t(B,c)}.
\end{equation}
If the distribution of $X$ is absolutely continuous with respect  to Lebesgue measure, so that $X$ has a density $f$,
the recipe \eqref{mixture RN} says that with $t$ given by \eqref{solution mass}, and  normalizing constant $k_c$ 
and function $h_c$ with domain $[1,c)$,
defined by   
\begin{equation}\label{def h}
   k_c := \int_{x=1}^c f(x) \,  t(x,c) \ dx,\ \     h_c(b) := \frac{1}{k_c}         f(b) \,  t(b,c),
\end{equation}
the measure $h$ governing the mixture has density $h_c$, so that for measurable $A \subset [1,c)$, $h(A) = \int_{{b \in A}} h_c(b) \ db$.
\end{theorem}
\begin{proof}
First, we must show that the distribution of $B$ was well-defined, i.e., that $\p(X \in [1,c))>0$.  Here we argue by contradiction:  if $\p(X \in [1,c))=0$, then Lemma  \ref{lemma times c} could be invoked, with $Y = e^{\sigma Z}, k=0$, to prove $X =^d Y$, a contradiction since $\p(Y \in [1,c))>0$.

Now write $Y$ for a random variable whose distribution is the mixture of the $\ell(b,c)$, governed by $h$.
We use the Dirac notation, that $\delta_x$ is unit mass at $x$, so that $\int g(z) \delta_x(dz) = g(z)$ for any measurable $g$.
Restricting our attention to $b \in [1,c)$, the Chihara--Leipnik distributions are then expressed as
$$
   \ell(b,c) = \sum_{n \in \BZ} \mu_{b,n}  \ \text{ where }  \mu_{b,n} := \frac{b^{-n} c^{-n^2/2}}{t(b,c)} \ \delta_{b c^n}.
$$
so that  $\mu_{b,n}$ is the measure $\ell(b,c)$ restricted to the interval $[c^n,c^{n+1})$ --- this uses $b \in [1,c)$.

Focus on the case $n=0$, so that $\mu_{b,0}$ is mass $1/t(b,c)$ at the point $b$.  
The specification of $h$ in \eqref{mixture RN} implies directly that the hypothesis \eqref{hyp k} holds  --- with $k=\p(X \in [1,c)) \times \e t(B,c)  $.  Hence by Lemma \ref{lemma times c}, we have $X =^d Y$.  

The argument for uniqueness is essentially the same:  suppose that $Y$
is a mixture of $\ell(b,c)$, governed by some probability measure $h$
on $[1,c)$, and that $X=^d Y$, not assuming that $h$ is given by \eqref{mixture RN}.
  Restricting the distributions of both $X$ and $Y$ to $[1,c)$, it is clear,
from  $\mu_{b,0} = 1/t(b,c) \ \delta_b$, that   the  Radon-Nikodym derivative $h(db)/ \p(X \in db \, | X \in [1,c))$
must be proportional to $t(b,c)$.  The recipe in 
\eqref{mixture RN} gives the unique constant of proportionality to make such an $h$ into a probability measure.
\end{proof} 

For each $c>1$, Lemma \ref{lemma Leipnik} says that the convex set
spanned by the Chihara--Leipnik distributions $\ell(b,c)$, $b \in
[1,c)$ is a \emph{subset} of the set $U_c$ of all solutions of \eqref{times c}.  Theorem 
\ref{thm ra} asserts that
$U_c$ 
is spanned by the $\ell(b,c)$,
so together with the obvious property that any single $\ell(b,c)$ is not a nontrivial mixture of other $\ell(b',c)$, one now knows that
 the  \emph{extreme} points of the set of solutions of \eqref{times c}
 \emph{are} the distributions $\ell(b,c)$, for $b \in [1,c)$.  
 For historical naming and perspective:  Choquet's Theorem states that for a convex compact subset of a normed space, every point can be represented as a mixture, governed by a probability measure, of extreme points; this probability measure need not be unique, even in the finite dimensional setting.  However, in the finite dimensional setting, uniqueness holds when the convex set is a simplex.  In honor of this, a convex set, for which the every point has a unique representation as a mixture of extreme points, is called a \emph{Choquet simplex}. 
 The
 additional information in Theorem \ref{thm ra} about the
 \emph{uniqueness} of $h$    is then summarized by saying that the set
 $U_c$ of solutions to \eqref{times c} forms a \emph{Choquet simplex}.
 
It is now natural to ask whether Stietljes' examples, with density given by \eqref{heyde both},
lie in this Choquet simplex.   

\begin{proposition}\label{prop stieltjes}
For every $\sigma>0$, integer $m$,  and  real $\delta \in [-1,1]$, the random variable $X$ with density given by Stieltjes' formula \eqref{heyde both} satisfies $X^* =^d cX$, with $c=\exp(\sigma^2)$, and hence $X$ satisfies \eqref{times c}.
\end{proposition}

\begin{proof}
For random variables with a density, the size bias scaling relation in \eqref{times c},  can be expressed in terms of the density, as follows.
First, when $X$ has density $f$, the scaled multiple $cX$ has density $(1/c)f(x/c)$.  Second, when $X$ has density $f$, and mean $a = \e X = \sqrt{c}$,  \eqref{sizebias-f} states that $X^*$ has density $(x/\sqrt{c}) f(x)$.   Hence, if $X$ has density $f$, mean $\sqrt{c}$, and 
\begin{equation}\label{times c density}
     \forall x \in (0,\infty), \ f(x/c) = x \sqrt{c} f(x),
\end{equation}
then $X$ satisfies \eqref{times c}.    Now it is clear that \eqref{times c density} holds for $f=h_{m,\delta}$ given by \eqref{heyde both}:   we have $c = e^{\sigma^2}$, and upon substituting $x/c$ for $x$, the lognormal factor $f_{0,\sigma^2}$ supplies the factor $x \sqrt{c}$, and the perturbation 
 factor $g_{m,\delta}$ supplies no change, since  dividing $x$ by $c$ causes $\log x$ to decrease by $\log c = \sigma^2$, so that the argument to the sine function, $2 \pi m \log x / \sigma^2$, goes down by  $2 \pi m$.
\end{proof}

To review:  both the lognormal  \emph{and} the examples given by
Stieltjes  are solutions of \eqref{times c}
and hence lie in the Choquet simplex $U_c$. 
Do all distributions having the lognormal moment sequence lie in this
simplex, i.e., does $U_c = V_c$? Berg \cite[Proposition
2.1]{berg}, proved $U_c \subsetneq V_c$ by exhibiting 
elements of $V_c \setminus U_c$.  These distributions can be described as the perturbations of
the  Chihara--Leipnik distribution in
\eqref{solution mass} by a factor of $(1 + s (-1)^n)$,
 for $s \in    [-1,1]$.  
 In detail, Berg showed that for any $c>1$,
\begin{equation}\label{berg}
b=\sqrt{c}, s \in \{-1,1\}, \ \ \p(X_s=bc^n)=(1+s(-1)^n)
b^{-n} c^{-n^2/2}/t(b,c), n \in \BZ
\end{equation} leads to $\e X_s^n = c^{n^2/2}$ for
$n \in \BZ$.  In particular, for $b=\sqrt{c}$ the  Chihara--Leipnik
distribution $\ell(b,c)$ is the midpoint of the line connecting the
distributions of $X_{-1}$ and $X_1$.  The construction is special to
$b =\sqrt{c}$ as the only value of $b \in [1,c)$ for which a line of 
distributions with moments  $\e X^n = c^{n^2/2}$ can be constructed,
with $\ell(b,c)$ as the midpoint.

\ignore{Going out on a limb, we conjecture that apart from the above
construction, the Chihara--Liepnik distributions are extreme points,
relative to the given moment sequence.  That is

\begin{conjecture}\label{conj times c}
For any $\sigma >0$,  with $c=\exp(\sigma^2)$, consider the set $V$ of probability measures on
$[0,\infty)$, serving as distributions for a nonnegative random
variable $X$ with  $\e X^n = c^{n^2/2}$ for $n=0,1,2,\ldots$.  This
set is a Choquet simplex, whose extreme points are precisely the
distributions $\ell(b,c)$ for $b \in [1,\sqrt{c}) \cup (\sqrt{c},c)$,
together with two additional distributions, those of   $X_{-1}$ and $X_1$,
as given by \eqref{berg}.
\end{conjecture}

Any lognormal distribution is infinitely divisible; 
for background and references on this, see Examples \ref{ex pareto} and \ref{ex lognormal}.
So, in the following conjecture, the emphasis is on the word \emph{only}.

\begin{conjecture}\label{conj logn}
For each $c>1$, in the Choquet simplex $U_c$
of solutions of \eqref{times c}, as described in Theorem \ref{thm ra},  the \emph{only} 
infinitely divisible distribution is the lognormal, that of $e^{\sigma Z}$ with $c=e^{\sigma^2/2}$.
\end{conjecture}

In the larger closed convex set $V_c$ of distribution sharing the
moments of the lognormal, the lognormal is \emph{not} the only
infinitely divisible element; see \cite{berg6}.
}   

We have shown that $U_c$ is a Choquet simplex;  the question as to whether
$V_c$ is a Choquet simplex is open.  We thank Christian Berg, private
communication, for this information and several references, and also for correcting two
erroneous conjectures from an earlier draft of our paper. 

\section{Size bias and  Skorohod embedding}\label{sect skorohod}

Skorohod's embedding theorem states that given a nonconstant mean zero random
variable $X$, there is a random time $T$ for
Brownian motion $(W_t)_{t \ge 0}$ such that $X =^d W_T$.  
We discuss Skorohod's proof as presented, for example, in
\cite{durrett,obloj}.  The  proof is based on  the construction of a joint distribution for a dependent
pair $(U,V)$ with $U,V \ge 0$  so that, with the pair independent of the Brownian motion,
the random time  $T := T_{U,V} := \inf\{t: W_t \notin [-U,V]\}$ yields $X =^d
W_T$. Since $\p(W_T=V|U,V)=U/(U+V))$, and the function $(u,v) \mapsto
u/(u+v)$ is nonlinear, it is somewhat surprising that a simple 
distribution of $(U,V)$ can satisfy $X =^d W_{T_{U,V}}$.  That
distribution,  specified in \cite{durrett,obloj} by the formula
$$
dH_\mu(u, v) = (v-􀀀u) 1(u
\le 0 \le v) \ \mu(du) \mu(dv) / \e X^+,
$$
where $\mu$ is the distribution of $X$, 
is the same\footnote{apart from a notational switch between $-u$ and
  $u$; we write $-u
  \le 0 \le v$ and they write $u \le 0 \le v$.} as the distribution \eqref{final recipe}
  below 
   in our
size bias treatment.
Display \eqref{kill} highlights how size bias overcomes the
nonlinearity of $(u,v) \mapsto
u/(u+v)$.
The excellent survey by  Ob\l\'{o}j \cite{obloj} should be consulted
for the history and connections to \emph{the potential of a measure.}

To define the joint distribution for $(U,V)$ in $[0,\infty)^2$, consider
 random variables $A,B$ with values in $[0,\infty)$ with distribution given by
\ignore{
$$
  \p(A \in C) = \frac{\p(-X \in C)}{\p(X<0)}, \ \   \p(B \in C) =
  \frac{\p(X \in C)}{\p(X>0)};
$$ 
}
$$
\cL(A) = \cL(-X|X<0), \ \ \ \cL(B) = \cL(X|X>0);
$$ 
since $X$ is nonconstant and
mean zero, both $p_-:=\p(X<0)>0$ and  $p_+ := \p(X>0)>0$, so 
the conditioning is elementary.  
Note that
\begin{equation}\label{champ}
\e A=\e X^-/p_-, \ \ \  \e B=\e X^+/p_+, \text{ and } \e X^- = \e X^+. 
\end{equation}
Write $p_0 := \p(X=0)$.
Since $A$ and $B$ have finite positive mean,   
the size biased distributions of $A^*$ and $B^*$ are well defined.
Couple so that  $A,A^*,B,B^*$ are independent.
The final recipe, writing $\delta_q$ for unit mass at the point $q$, is
\begin{equation}\label{final recipe}
\cL(U,V) = p_+ \ \cL(A^*,B) + p_0 \ \delta_{(0,0)} + p_-\ \cL(A,B^*),
\end{equation}
and then take $(U,V)$ to be independent of the Brownian motion $W$.

To prove that \eqref{final recipe} and $T=T_{U,V}$ achieve $X =^d
W_T$, first consider the case where $\p(X=0)=0$. 
Given a bounded measurable function $h: \BR \to \BR$, conditioning on
$U,V$ and using the exit distribution for Brownian motion from the
interval $[-u,v]$ we  have
\begin{equation}\label{g from h}
\e ( h(W_T)|U=u,V=v ) = 
h(-u)\frac{v}{u+v} + h(v)\frac{u}{u+v}  =: g(u,v).
\end{equation}
Next, since we are in the case where $p_- + p_+=1$, 
using \eqref{champ} we have 
\beas
p_- = \frac{\e B}{\e A+ \e B},  \ \ \ p_+ = \frac{\e A}{\e A+ \e B}.
\enas
The size bias relation for processes from Section \ref{sect basics one}, together with the independence of $A,B$,  justifies the transition from
line 2 to line 3 below:  for any
bounded measurable $g: \BR^2 \to \BR$,
\begin{eqnarray}
 \e g(U,V)&=& p_+ \  \e g(A^*,B)+ p_- \  \e g(A,B^*)  \nonumber \\
       &=& \frac{\e A \ \e g(A^*,B)+ \e B \ \e g(A,B^*)}{\e(A+B)} \nonumber\\
       &=& \frac{\e(A g(A,B)) + \e (B g(A,B))}{\e(A+B)}  \label{kill} \\
       &=& \frac{\e((A+B) \ g(A,B))}{\e(A+B)}. \nonumber
\end{eqnarray}
Using this identity for our function $g$ defined in \eqref{g from h}, and using  the independence
of $A$ and $B$ to go from line 3 to line 4,  we have
\beas
\e h(W_T)      &=& \e g(U,V) \\
               &=& \frac{\e((A+B)\ g(A,B))}{\e(A+B)} \\
               &=&\frac{\e \left( h(-A)B+h(B)A \right)}{\e(A+B)}\\
               &=& \frac{\e B}{\e (A+B)} \ \e h(-A) + \frac{\e A}{\e
                 (A+B)} \  \e h(B)\\
               &=& p_- \ \e h(-A) + p_+ \ \e h(B)\\
               &=& \e(h(X)|X<0) \ p_- + \e (h(X)|X>0) \ p_+ \\
               &=& \e h(X),
\enas
and hence ${\mathcal L}(W_T)={\mathcal L}(X)$, as claimed.

That  $X =^d
W_T$ in the general situation, allowing $\p(X=0) \in (0,1)$, is easily seen, 
since the distribution of $X$ is then a mixture of pointmass at zero, and the distribution of $X$ conditional on $X \ne 0$, and the recipe \eqref{final recipe} is the corresponding mixture of  pointmass at (0,0)  and the distribution of $(U,V)$ treated above.

\section{Size bias and infinite divisibility}\label{sect inf div}
Paul L\'evy's theory of infinitely divisible distributions
is celebrated;  see
any of \cite{billingsley,chung,feller2,kallenberg} for introductory treatments,
or \cite{applebaum,bertoin,sato} for advanced treatments.  
For the special case of nonnegative
random variables with finite mean, size  bias provides an easy handle on the theory.

\subsection{Steutel revisited}

\begin{theorem}\label{thm inf div}
Suppose $X$ can be size biased, i.e., $X \ge 0$ and $a := \e X \in
(0,\infty)$.  If $X$ is infinitely divisible, then there exists a
distribution for $Y$ 
such that
\begin{equation}\label{indep}
X^*=^dX+Y,  \ \ \mbox{ and } X,Y \mbox{ are independent}.
\end{equation} 
Conversely, given that $X$ can be size biased, and that  \eqref{indep}
holds for some $Y$, then $X$ is infinitely divisible.

In either case,  the distribution of $Y$ is
unique, and $\p(Y \ge 0)=1$.
\end{theorem}

Remark: In \cite{steutel} (see also \cite{steutel70}), F. Steutel
  shows that a cumulative distribution function $F$ on $[0,\infty)$ is
  infinitely divisible iff it satisfies
$$
\int_0^x u dF(u) = \int_0^x F(x-u) dK(u)
$$
for a non-decreasing $K$.  Our decomposition \eqref{indep} is clearly
a consequence of his integral formula, though he does not use the
language of size biasing --he does not, in fact, assume that $F$ has finite
mean--
and his proof proceeds by way of the Levy representation formula, which we will
derive instead as a corollary of  \eqref{indep}.  Steutel's result is also presented in 
Sato \cite{sato}, Theorem 51.1, as well as in the book \cite{steutel-book} by Steutel and van Harn.

\begin{proof}
We begin by assuming that $X$ is infinitely divisible, which by definition means that 
 for each $n$ there exists a distribution
such that if $X_1^{(n)},\ldots,X_n^{(n)}$ are i.i.d. with this distribution, then
\begin{equation}\label{inf div}
X \indist X_1^{(n)}+\cdots + X_n^{(n)}.
\end{equation} 
Then 
by (\ref{iid})
\begin{equation}\label{inf div sum}
X^* =^d (X-X_1^{(n)}) + (X_1^{(n)})^*,
\end{equation}
with $X-X_1^{(n)}$ and $(X_1^{(n)})^*$ independent.

It is obvious that, with probability 1,  $X_1^{(n)} \ge 0$, since \eqref{inf div}
gives
$ (\p(X_1^{(n)}<0))^n  \le \p(X<0)=0 $.
Next, $\e |X_1^{(n)}| = \e X_1^{(n)} =a/n \to 0$ as $n \to \infty$
implies that $X_1^{(n)} \rightarrow 0$
in $L_1$ and hence in probability. Hence
$X-X_1^{(n)}$
converges in distribution to $X$.

Next, the family of random variables $(X_1^{(n)})^*$ is \emph{tight},
because 
given $\epsilon >0$, there is a $K$ such that $\p(X^* >K)<\epsilon$, and
by
(\ref{inf div sum}), for
all $n$, $\p((X_1^{(n)})^*>K) \leq \p(X^*>K)$.  Thus, by Helly's theorem,
there exists a
subsequence $n_k$ of the $n$'s along which $(X_1^{(n)})^*$ converges in distribution,
say $(X_1^{(n_k)})^* \todist Y$.
As  $n \rightarrow \infty$ along this subsequence, the pair  $(X-X_1^n,(X_1^n)^*)$
converges jointly to the pair $(X,Y)$ with $X$ and $Y$ independent.  From
$X^* =^d (X-X_1^{(n_k)}) + (X_1^{(n_k)})^* \todist X+Y$ as $k \rightarrow \infty$
we conclude that $X^* \indist X+Y$, with $Y \geq 0$, and  $X,Y$
independent.  This completes the proof that if $X$ is infinitely
divisible, then it satisfies \eqref{indep}.  

That the law of $Y$ in \eqref{inf div} is \emph{unique} requires a
little work; we will need to know that   the characteristic function $\phi$ for $X$ satisfies $\phi(u) \ne 0$ for
all real $u$.  Once we have this,   
uniqueness is easy:  from \eqref{phi'} and 
\eqref{indep}, writing $\phi_Y$ for the characteristic function of
$Y$, we have two expressions for $\phi_{X^*}(u)$, hence
\begin{equation}\label{char fns}
  \frac{1}{i \, \e X} \ \phi'(u)   = \phi(u) \
  \phi_Y(u).
\end{equation}
This determines $\phi_Y(u)$, provided we know that $\phi(u) \ne 0$.

The characteristic function of \emph{any} infinitely divisible $X$ has
$\phi(u) \ne 0$ for all $u$: 
 Feller \cite[p. 500 and pp. 555--557]{feller2}, 
and Chung \cite[Theorem 7.6.1]{chung}, 
give straightforward proofs. However, under the hypothesis that 
  \eqref{indep} holds and $\e X$ is finite, there
is a simpler proof, as follows.
\emph{
Suppose that $\phi(u) \ne 0$ for all $u \in (-t,t)$, for some $t>0$.
From equation \eqref{char fns}, for  $u \in (-t,t)$
$$
  (\log \phi(u))' = \frac{ \phi'(u)}{ \phi(u)} = i \, \e X \ \phi_Y(u) , \text{
  hence } |(\log \phi(u))'| \le \e X.
$$
Since $\phi$ is continuous with $\log \phi(0)=0$, it follows that for
all  $u \in [-t,t]$, $|\log \phi(u)| \le  t \e X < \infty$.  If it
were the case that $\phi(u)=0$ for any $u$, we could take $t = \inf \{
|u|: \phi(u)=0 \} < \infty$ to get a contradiction.\footnote{As to the
  validity of taking logarithm, log continues
  uniquely along paths avoiding zero; see, e.g., \cite[pp. 554--5]{feller2}, and  \cite[p. xv line -7 and Thm. 7.6.2]{chung}.} }
  
Finally, we prove the converse statement, that \eqref{indep} implies infinite divisibility.
Starting
with the assumption \eqref{indep}, we have \eqref{char fns}, which
--- with details given in the next section ---
lets us solve for $(\log \phi(u))'$, and integrate, to get
\eqref{comp pois 1} 
below. 
That \eqref{comp pois 1} is the characteristic
function of an infinitely divisible distribution is well-known, but to
review, for the sake of a self-contained proof: the function in
\eqref{comp pois 1}
can be expressed as the limit of characteristic functions
of random variables with compound Poisson distribution, 
as in \eqref{comp pois phi}, and scaling
all the Poisson parameters down by a factor of $n$, and then taking
the limit, we get the distribution for 
 the $n^{\rm th}$ convolutional root  $X_1^{(n)}$
for use in \eqref{inf div}.
\end{proof}

\subsection{The \levy representation}

We continue to work with an $X \ge 0$ with 
$a := \e X \in (0, \infty)$, assuming also that $X$ is  infinitely divisible, or
equivalently, that $X$ satisfies \eqref{indep}.  Using \eqref{char fns}, 
\begin{equation}\label{log deriv}
(\log \phi(u))' =  \frac{\phi'(u)}{\phi(u)} = a \, i \ \phi_Y(u)
\end{equation}
and since $\phi(0)=1$ with $\log \phi(0)=0$, we get
$$
  \log \phi(u) = a \, i \ \int_{t=0}^u  \phi_Y(t) \ dt.
$$
Let $\alpha$ be the distribution of $Y$ in \eqref{inf div}, so $\alpha$ is a
probability measure on $[0,\infty)$.   We have
$$ 
\int_{t=0}^u  \phi_Y(t) \ dt = \int_{t=0}^u \int_y e^{ity} \alpha(dy) \ dt=
\int_y \int_{t=0}^u e^{ity} \ dt \ \alpha(dy) 
$$
with the interchange justified by Fubini.
We have 
$$ 
\int_{t=0}^u e^{ity} \ dt = \left\{
\begin{array}{ll}   
(e^{iuy}-1)/(iy)  & \text{ if } y>0 \\
 u  & \text{ if } y=0
 \end{array}
 \right.
$$
 Combining the three previous displayed equations, the characteristic function $\phi$ for $X$ may be expressed as
\begin{equation}\label{comp pois 1}
 \phi(u) =\exp \left( a \left( i u \,  \alpha(\{0\}) +\  \int_{(0,\infty)} \frac{e^{iuy}-1}{y} \
   \alpha(dy) \right) \right).
\end{equation}
To review,  $a \in (0,\infty)$, $\alpha$ is the probability distribution of a nonnegative random variable $Y$, and $\phi(u)$ is the characteristic function of a random variable $X$, with $a=\e X$, and, with $X,Y$ independent, $X^*=^d X+Y$.   We have derived  \eqref{comp pois 1} 
 under the assumption that \eqref{indep} holds.  However, given $a \in (0,\infty)$, and a probability distribution for a nonnegative random variable $Y$, it can be seen that 
\eqref{comp pois 1} is the characteristic function of a random variable $X$, by taking distributional limits of the discrete compound Poisson
sums in \eqref{comp pois phi}. Then, working back through \eqref{log deriv}, one sees easily that $\e X = a$ 
and, with $X,Y$ independent, $X^*=^d X+Y$.

The calculation above, combined with Theorem \ref{thm inf div}, is
summarized in the next theorem.

\begin{theorem}\label{thm inf div and Levy}
Suppose $X$ can be size biased, i.e., $X \ge 0$ and $a := \e X \in
(0,\infty)$.  If $X$ is infinitely divisible, then there exists a
distribution for $Y$ 
such that
\[
X^*=^dX+Y,  \ \ \mbox{ and } X,Y \mbox{ are independent}.
\] 
Conversely, given that $X$ can be size biased, and that  \eqref{indep}
holds for some $Y$, then $X$ is infinitely divisible.

In either case,  the distribution of $Y$ is
unique, $\p(Y \ge 0)=1$, and $X$ has characteristic function given by  \eqref{comp pois 1}.
\end{theorem}

\begin{corollary}\label{cor poisson}
If $X$ is a nonnegative random variable with $\lambda := \e X \in (0,\infty)$, and
$X^* =^d X+1$, then $X$ is Poisson($\lambda$).
\end{corollary}

A natural way to rewrite \eqref{comp pois 1}, motivated perhaps by the two expressions in
\eqref{comp pois phi}, is to absorb the $1/y$ into the measure $\alpha(dy)$.   Writing $\alpha_0$ for the constant $\alpha(\{0\}) = \p(Y=0)$ in 
\eqref{comp pois 1}, this gives
 \begin{equation}\label{comp pois 2}
 \phi(u)  =
\exp \left( a \left( i u \alpha_0+  \int_{(0,\infty)} \left(e^{iuy}-1 \right) \  \gamma(dy) \right) \right) .
\end{equation}
Here $\gamma$ is a nonnegative measure on $(0,\infty)$, with $\gamma(dy) /\alpha(dy) = 1/y$, and this allows a \emph{broader} class than
\eqref{comp pois 1}.  To get $\e X < \infty$, there is the additional requirement that $\int_{(0,\infty)} y \, \gamma(dy) < \infty$   --- this is the price one pays for being able to size bias.   Regardless of whether $\e X = \infty$ or $\e X < \infty$,   the nonnegative measure $\gamma$ can have infinite mass, due to mass near zero, and 
the  requirement, to get a nonnegative infinitely divisible $X$, allowing $\e X= \infty$,  is that $\int_{(0,\infty)} (1 \wedge y) \gamma(dy) < \infty$.  Examples \ref{ex dickman} and \ref{ex buchstab} illustrate this, where,  
in both cases, $\alpha$ is a uniform distribution 
on an interval, and $\e X < \infty$.

    We read \eqref{comp pois 2} as:  the random variable $X$ is the constant $a \, \alpha_0$, plus the sum of arrivals, in the Poisson process on $(0,\infty)$ with intensity measure $a \, \gamma$.
Formula \eqref{comp pois 2}\footnote{Or any of its cousins, such as the Laplace transform version --- since the characteristic function $\phi(\cdot)$, moment generating function $M(\cdot)$, and Laplace transform $L(\cdot)$ are essentially the same, in detail  $\phi(u):= \e e^{iuX}$, $M(\beta):= \e e^{\beta X}$, $L(t) := \e e^{-tX}$, allowing the formal substitutions $iu = \beta = -t$.} is called the L\'evy--Khintchine formula in the survey paper on subordinators \cite{bertoin-sub},  the one difference being that the random variable $X$ representing the value of the subordinator at time $a$ is also allowed to have $\p(X=\infty) = 1-\exp(-ka)>0$, where $k$ is called the  \emph{killing rate}.\footnote{And when there is killing, then the Laplace transform \emph{is preferable} to the characteristic function; see the previous footnote.}
 
\ignore{
\levy \emph{canonical measures} include the probability distribution $\alpha$, and the pair $(\alpha_0,\gamma)$ of \eqref{comp pois 1}
and \eqref{comp pois 2} as the \emph{special case} where $X \ge 0$ and $\e X < \infty$;  our goal was to show how natural and easy
this subject is, under these additional restrictions, thanks to size bias.  For the general case, we refer to
\cite{applebaum,bertoin,sato}.  
} 
 
\subsection{The size bias equation}

When $X,Y$ are both  discrete or both absolutely continuous, it is worth highlighting
how \eqref{sizebias-f}, together with \eqref{indep}, 
yields a simple relation satisfied by 
the mass functions or densities.   Sato \cite{sato} Section 51,
especially Corollary 51.2, already highlights these relations, though
of course without referring to them as being size bias relations.

In the discrete case, if \eqref{indep} holds,
then 
$f_{X^*}$ is the convolution
of $f_X$ and $f_Y$: \ $f_{X^*}(x) = \sum_y f_X(x-y) f_Y(y)$, 
and
combining with
(\ref{sizebias-f}) yields, for all $x>0$,
\begin{equation}\label{discrete convolve}
f_X(x) = \frac{a}{x} \sum_y f_X(x-y) f_Y(y).
\end{equation}
A common special case is that $Y$ is supported on the positive
integers, and $X$ on the nonnegative integers, so that 
considering $f_Y$ as known, and $f_X$  to be found, 
the homogeneous system
\eqref{discrete convolve} specifies a recursion:  starting from
$f_X(0)=c$, for $m=0,1,2,\ldots$,
\begin{equation}\label{discrete convolve recursion}
     f_X(m+1) = \frac{a}{m+1} \sum_{0 \le i \le m} f_X(i) f_Y(m+1-i),
\end{equation}
and the initial value $c$ is determined by $1= \sum_{i \ge 0}  f_X(i)$.
Furthermore,  from  \eqref{comp pois phi} and  \eqref{comp pois 1} we know
that $X = \sum_{i \ge 1} i Z_i$ with $Z_i$ independent
Poisson($\lambda_i)$, $\lambda := \sum \lambda_i < \infty$,
$f_Y(i) = i \lambda_i /a$, 
hence $f_X(0)= P(Z_1=Z_2=\cdots =0)=e^{-\lambda}$.
The relation \eqref{discrete convolve recursion} was used in \cite{at92}, where it was referred to as a result from
\cite{pour}.  The situation with $X = \sum_{1 \le i \le n} i Z_i$ with $Z_i$ independent
Poisson($\lambda_i)$ is universal to \emph{combinatorial assemblies}; here $X$ is usually denoted as $T_n$, and conditional on the event $(T_n=n)$ one has a labelled combinatorial object of total size $n$, in which there are $Z_i$ components of size $i$, jointly for $i=1$ to $n$.  See \cite{IPARCS,abt}. 
 
Likewise, in the absolutely continuous case, where $X$ and $Y$ have
densities, if \eqref{indep} holds, then
$f_{X^*}$ is the convolution
of $f_X$ and $f_Y$: \ $f_{X^*}(x) = \int_y f_X(x-y) f_Y(y) \ dy$.  Combined with
(\ref{sizebias-f}), this says that for all $x>0$,
\begin{equation}\label{cont convolve}
f_X(x) = \frac{a}{x} \int_y f_X(x-y) f_Y(y) \ dy .
\end{equation}

\subsection{Examples of infinitely divisible distributions for nonnegative random variables}

Of course, the \levy representation \eqref{comp pois 2} yields \emph{all} examples of nonnegative infinitely divisible distributions.    However,
\emph{recognizing} when a given distribution for $X$ takes the form \eqref{comp pois 1} 
or \eqref{comp pois 2} remains a nontrivial problem.   
We present our favorite examples in which Theorem \ref{thm inf div}
provides a convenient criterion, and we will use the notation from
Theorem \ref{thm inf div}, in particular \eqref{indep}.

\subsubsection{Discrete examples}

\bex\label{ex poisson}
$\p(Y=1)=1$;  $X$ is Poisson(a).
\eex

\bex\label{ex geometric}
For $p \in (0,1)$,  $\p(Y=k)= (1-p)^k/k$.
When $a=(1-p)/p$,   $X$ is geometric, with $\p(X=n)=(1-p)^n p, n\ge 0$.  When $ap/(1-p)$ is a positive integer,  $X$ is negative binomial.   

The infinite divisibility of geometric and negative binomial
distributions plays a key role in estimates comparing logarithmic
combinatorial structures with their limits; see \cite{abt}. The
compound Poisson representation of the geometric is the starting point
for a coupling, in \cite{budalect}, showing that a random integer may
be chosen uniformly from 1 to $n$, on the same probability space with
a Poisson-Dirichlet process $(L_1,L_2,\ldots)$, so that if $P_i$ is
the $i^{\rm th}$ largest factor of the random integer,\footnote{with the convention that $P_i=1$ when $i$ exceeds the number of prime factors, including multiplicities} then $\e \sum_{i \ge 1} |\log P_i - (\log n) L_i| = O(\log \log n)$.
This construction is analogous to Skorohod embedding: it starts with the continuum limit process --Poisson-Dirichlet instead of  Brownian motion-- and constructs the nearby (in the limit)  
discrete random object --the random integer expressed as a product of primes instead of a 
random walk-- as a deterministic function of the continuum limit process, together with a small amount of auxiliary randomization. 
\eex

A necessary and sufficient condition for a nonnegative integer valued random variable to be infinitely divisible is given in \cite{katti1}, and   
a useful \emph{sufficient} condition is given in \cite{katti2}. 
The sufficient condition is \emph{log-convexity}: the support of $X$ is the nonnegative integers, and  for all $n \ge 1$,
$\p(X=n-1) \p(X=n+1) \ge \p(X=n)^2$.    Example \ref{ex poisson} shows that the sufficient condition of log-convexity is not necessary --- any Poisson distribution is log-\emph{concave}, rather than log-\emph{convex}.
See \cite{alw} for a discussion of how the sufficiency of log-convexity is 
perhaps attributable to Kaluza, \cite{kaluza}.   Of course, for any constant $c$, $X$ is infinitely divisible if, and only if, $c+X$ is infinitely divisible;  this remark is often used with $c=\pm 1$.
There are several famous discrete distributions that can be seen to be infinitely divisible via log-convexity; some
 examples of this type are given in \cite{katti2}, and two of our favorite examples are the following:

\bex The zeta distributions: For  $s>0$, $ \p(X=n) = n^{-s}/\zeta(s)$,  $n\ge 1$. 
\eex

\bex  The simplest power law, $\p(X \ge n) = 1/n$ for $n \ge 1$.
\eex

\subsubsection{Continuous examples}


\bex  $Y$ is exponential, with $\p(Y>t)=e^{-t}$ for $t \ge 0$.   When $a=1$, $X=^d Y$, and $X^* =^d X+Y$ is the sum of two independent copies of $X$, as observed in Section \ref{sec:renewal}  on the waiting time paradox. For positive integers $a$, $X$ is the time of the $a^{\rm th}$ arrival in a standard Poisson process.  For general $a > 0$, $X$ has the Gamma distribution, with shape parameter $a$.
\ignore{
\footnote{Scaling is a trivial change, but a source of much confusion:  taking $X/r$ for large $r$ yields a smaller value of  $X$,  and the parameter $r$ is called the arrival rate; other authors take $cX$, so that  $\e X = ac$, and the parameter $c$ is the mean of $Y$.  But alas, either $r$ or $c$ might be called \emph{the scale parameter} of the corresponding Gamma distribution.}
} 

In the \levy representation \eqref{comp pois 2} for the characteristic function of the Gamma random variable $X$,  we have  $\gamma(dy) = e^{-y}/y \ dy$.  This measure $\gamma$, or the increasing process it governs, is also known as the Moran subordinator, and used to construct the Poisson-Dirchlet process;  see \cite{kingman}.
\eex

\bex\label{ex pareto}
Pareto distributions, of the form $\p(X>t) = (1+t)^{-\alpha}$, $\alpha>0$.  

This is the example for which Thorin \cite {thorin-pareto} first developed his theory of generalized Gamma convolutions, which is a subclass of the infinitely divisible distributions for positive random variables.  See \cite{bondesson}, as well as  \cite {bondesson-thorin}.
\eex

\bex\label{ex lognormal}

The  lognormal distributions.  Again, this is from Thorin in 1977, \cite{thorin}, and his proof is based on a generalized Gamma convolution.  
 \eex

\bex  Distributions with a log-convex density.  

Taking limits of discrete distributions on the nonnegative integers with log-convex pointmass function, Sato \cite[Theorem 5.1.4]{sato}  shows that if $X$ has a density $f$ on $(0,\infty)$, 
such that $\log f$ is  convex on $(0,\infty)$, then $X$ is infinitely divisible.  This also shows that the Pareto distributions are infinitely divisible!
\eex

The next two examples, Examples \ref{ex dickman} and \ref{ex buchstab}, arise by taking $Y$ in \eqref{indep} to be uniformly distributed on a bounded interval of nonnegative numbers.  Up to scaling, \emph{any} such $Y$  is 
either uniformly distributed on $(0,1)$, or else on $(b,1)$ for some $0<b<1$.   In the former case,
$X$ has an absolutely continuous distribution,  and the latter case the distribution of $X$
has an atom and an absolutely continuous part.

\bex\label{ex dickman}  
$Y$ is uniform $(0,1)$, leading to Dickman's function $\rho$, and its convolution powers.

In \eqref{indep}, take $Y$ to be the standard uniform random variable
on $(0,1)$.  Then
 \eqref{comp pois 1} specializes to 
\begin{equation}\label{comppois dickman}
 \phi_X(u) =\exp \left( a \  \int_0^1 \frac{e^{iuy}-1}{y} \
   dy \right),
\end{equation}
and \eqref{cont convolve} specializes to
\begin{equation}\label{dickman int}
  f_X(x) = \frac{a}{x} \int_{y=0}^1 f_X(x-y) \ dy
= \frac{a}{x} \int_{x-1}^x f_X(z) \ dz.
\end{equation}

Here as always, $a= \e X$;  the choice $a=1$ yields
$f_X(x)=e^{-\gamma}\rho(x)$, where $\rho$ is Dickman's function,
of central importance in the study of integers
without large prime factors; see \cite{tenenbaum} and  \cite[Section 4.2]{abt}.  For the general
case $a \in (0,\infty)$,  the density $f_X$ is a  ``convolution power of Dickman's function,''
normalized to be a probability density;  see \cite{hensley} .
\eex

\bex \label{ex buchstab}   $Y$ is uniform $(b,1)$ for $0<b<1$, leading to 
Buchstab's function $\omega$, and the limit probability for \emph{logarithmic}
structures to have all parts in a range excluding small parts, or both small and large parts.   

Now
\eqref{comp pois 2} becomes
\begin{equation}\label{comppois buchstab}
 \phi_X(u) =\exp \left( \frac{a}{1-b} \  \int_b^1 \frac{e^{iuy}-1}{y} \
   dy \right),
\end{equation}
with $0 < b <1$.  
Unlike Example \ref{ex dickman},  $X$ is no longer absolutely
continuous, since $\p(X=0)=b^{ a/(1-b)}>0$.

This computation of $\p(X=0)$ is easy to understand, by viewing 
  \eqref{comp pois 2} as the specification that
$X$ is the sum of the arrivals in the Poisson process with
arrival intensity measure $ a \, \gamma$, where  $a \gamma(dy)=  a/(1-b) \ 1(b < y <
1) \ dy/y$. The expected number of arrivals in this Poisson process
 is $\lambda = \int_b^1  a/(1-b) \
dy/y$, and of course $\p(X=0)=e^{-\lambda}$.    See \cite{scale-invariant}.

 The size bias squation, which was  \eqref{dickman int} for the case $b=0$,  is more complicated with $0<b<1$: the distribution of $X$ has pointmass $b^{ a/(1-b)}>0$, and a defective density $f_X$ whose support is $\cup_{k \ge 1} [kb,k]$.  The size bias equation 
 obtained by combining \eqref{sizebias-f} with \eqref{indep} takes the form:  for $x>0$,
\begin{equation}\label{buchstab int}
  f_X(x) =  \frac{a}{x} \left( b^{ a/(1-b)} \, \frac{1(b<x<1)}{1-b} +  \int_{y=b}^1 \frac{f_X(x-y)}{1-b} \ dy \right).
\end{equation}

We briefly explain the natural importance of
Example \ref{ex buchstab}.
 Let  $f_X^{(b)}$ be the density of 
  $X$, for   $0<b<1$ and $a=1-b$.  This density arises in the study of random permutations; see
   \cite[Section 4.3]{abt}.    Directly, $f_X^{(b)}(1)$ governs the asymptotic probability that a random permutation of $n$ objects has only cycles of length at least $bn$.
  Scale invariance also leads, for fixed $b \in (0, 1)$, to $f_X^{(b)}(u)$ governing the asymptotic probability that a random permutation on $n$ objects has only cycles with lengths in $(bn/u, n/u)$, for any $u > 1$.
  Scale invariance also leads to $\omega(u) = f_X^{(1/u)}(1)$, with Buchstab's function $\omega$ governing integers free of small prime factors;  see \cite{tenenbaum}.
       \eex

\bibliographystyle{plain}
\bibliography{csbbib}

\ignore{ OLD BIBLIOGRAPHY

}   

\end{document}